\documentclass[leqno, 12pt]{article}
\usepackage{amsmath,amsfonts,amsthm,amssymb,indentfirst}
\usepackage{xy} \xyoption{all}
\usepackage[pagebackref,colorlinks]{hyperref}
\hypersetup{linkcolor=blue, urlcolor=blue, citecolor=red}
\usepackage[capitalise]{cleveref}

\newtheorem{theorem}{Theorem}
\newtheorem{lemma}[theorem]{Lemma}
\newtheorem{definition}[theorem]{Definition}
\newtheorem{corollary}[theorem]{Corollary}
\newtheorem{proposition}[theorem]{Proposition}

\theoremstyle{definition}
\newtheorem{example}[theorem]{Example}

\newcommand{\im}{\mathrm{Im}}
\newcommand{\dom}{\mathrm{Dom}}
\newcommand{\pth}{\mathrm{Path}}
\newcommand{\clpth}{\mathrm{ClPath}}
\newcommand{\so}{\mathbf{s}}
\newcommand{\ra}{\mathbf{r}}

\newcommand{\I}{\mathcal{I}}
\newcommand{\J}{\mathcal{J}}
\newcommand{\B}{\mathcal{S}}
\newcommand{\SR}{\mathcal{O}}
\newcommand{\T}{\mathcal{T}}
\newcommand{\PT}{\mathcal{PT}}

\newcommand{\Z}{\mathbb{Z}}

\newcommand{\N}{\mathbb{N}}
\newcommand{\M}{\mathbb{M}}

\begin{document}

\title{Conjugacy and Least Commutative Congruences in Semigroups}

\author{Zachary Mesyan}

\maketitle

\begin{abstract}
Given a semigroup $S$ and $s,t \in S$, write $s \sim_p^1 t$ if $s=pr$ and $t=rp$, for some $p,r \in S \cup \{1\}$. This relation, known as ``primary conjugacy", along with its transitive closure $\sim_p$, has been extensively used and studied in many fields of algebra. This paper is devoted to a natural generalization, defined by $s \sim_s^1 t$ whenever $s=p_1\cdots p_{n}$ and $t=p_{f(1)}\cdots p_{f(n)}$, for some $p_1, \dots, p_n \in S \cup \{1\}$ and permutation $f$ of $\{1, \dots, n\}$, together with its transitive closure $\sim_s$. The relation $\sim_s$ is the congruence generated by either $\sim_p^1$ or $\sim_p$, and is moreover the least commutative congruence on any semigroup. We explore general properties of $\sim_s$, discuss it in the context of groups and rings, compare it to other semigroup conjugacy relations, and fully describe its equivalence classes in free, Rees matrix, graph inverse, and various transformation semigroups.

\medskip


\end{abstract}

\section{Introduction} \label{intro}

Given elements $s$ and $t$ in an algebraic structure $S$ (i.e., a set with a binary operation), write $s \sim_p^1 t$ if $s=pr$ and $t=rp$, for some $p,r \in S \cup \{1\}$. This relation, and its transitive closure $\sim_p$, have been used repeatedly in all corners of algebra, and beyond. For example, if $S$ is a group, then $\sim_p^1 \, = \, \sim_p$ is simply the familiar conjugacy relation. For this reason $\sim_p$ has been extensively studied in the literature on semigroups, as a possible generalization of conjugation to that context, where it is known as the \emph{primary conjugacy} relation. Various other generalizations have been proposed (see~\cite{ABKKMM,AKKM1} and the references therein), but this one is perhaps the most well-known. (See~\cite{AKKM1,KM2} for overviews of the history of $\sim_p$ in semigroups, which goes back as far as at least the 1950s.) In the $C^*$-algebra literature $\sim_p^1$ is known as the \emph{Murray--von Neumann equivalence}, and is used on projections in the process of constructing $K_0$-groups~\cite[Chapters 2 and 3]{RLL}. In symbolic dynamics $\sim_p^1$ and $\sim_p$ are known as the \emph{elementary shift equivalence} and \emph{strong shift equivalence}, respectively, and are used on certain integer-valued matrices in order to describe conjugate edge shifts~\cite[Section 7.2]{LM}. In ring theory $\sim_p$ is closely tied to commutators and trace maps~\cite{MV}, and has been studied as a measure of commutativity~\cite{AbL,AlL}. 

It seems apparent that the wide-ranging usefulness of the relations $\sim_p^1$ and $\sim_p$ can be typically attributed to their ability to measure and force commutativity in an algebraic structure. With that in mind, it is natural to study a more general variation on the same idea. Specifically, for elements $s$ and $t$ in an algebraic structure $S$ write $s \sim_s^1 t$ if $s=p_1\cdots p_{n}$ and $t=p_{f(1)}\cdots p_{f(n)}$, for some $p_1, \dots, p_n \in S \cup \{1\}$ and permutation $f$ of $\{1, \dots, n\}$, and let $\sim_s$ denote the transitive closure of $\sim_s^1$. The relations $\sim_s^1$ and $\sim_s$ have indeed been studied in the context of rings in~\cite{AlL,LN}, while a similar relation on semigroups is discussed in~\cite{CCT}. Our goal here is to describe the precise relationship between $\sim_p$ and $\sim_s$, and then to explore $\sim_s$ (and $\sim_s^1$) in the general setting of semigroups.

It turns out that $\sim_s$ is nothing more than the semigroup congruence generated by either $\sim_p^1$ or $\sim_p$ (but not smaller relations--Proposition~\ref{free-small-gen}), and is moreover the least congruence that produces a commutative quotient semigroup (Theorem~\ref{least-comm}). Additionally, just as $\sim_p$ reduces to the usual conjugacy in any group, $\sim_s$ has a natural interpretation in groups as well. Specifically, given elements $s$ and $t$ in a group $G$, we have $s \sim_s t$ if and only if $st^{-1}$ belongs to the commutator subgroup $[G,G]$ of $G$ (Corollary~\ref{group-sym}). So one might argue that $\sim_s$ completes the information about commutativity captured by $\sim_p$. Moreover, while it does not quite generalize conjugacy in groups, it does share various properties with $\sim_p$ and other conjugacy relations proposed for semigroups (Lemmas~\ref{Leroy-lemma2} and~\ref{inv-semi-lem}).

In order to get a sense for how $\sim_s$ can behave, most of the paper is devoted to describing $\sim_s$-equivalence classes in various standard types of semigroups, and comparing them to equivalence classes under $\sim_p$ and other existing conjugacy relations. In some cases $\sim_s$ can be computed rather quickly from previously known results (such as descriptions of $\sim_p$) and the fact that it is the least commutative congruence, but in other cases describing the relation can be quite challenging. In particular, we completely classify $\sim_s$-equivalence classes in free semigroups (Proposition~\ref{free-semi-prop}), Rees matrix semigroups (Corollaries~\ref{rees-0-sym} and~\ref{rees-no-0-sym}), graph inverse semigroups (Theorem~\ref{graph-semi-sim}), full transformation monoids, partial transformation monoids, symmetric inverse monoids (Proposition~\ref{full-transf-sim}), and injective function monoids (Theorem~\ref{inj-sym-thrm}). We also give a partial description of $\sim_s$ in the monoid of all surjective functions on a set (Theorem~\ref{surj-sym}). Along the way, we completely classify $\sim_p^1$- and $\sim_p$-equivalence classes in some semigroups where they had not been previously described in full generality, namely Rees matrix semigroups (Theorem~\ref{rees-prim}) and injective function monoids (Theorem~\ref{equiv-rel}).

Finally, in an appendix, we attempt to explain precisely the special nature of $\sim_p$ in semigroup rings. Just as $\sim_s$ relates exactly the elements conflated by homomorphisms with largest possible commutative images, $\sim_p$ relates exactly the elements conflated by certain trace maps. More specifically, we show that for any two elements $s$ and $t$ in a semigroup, we have $s\sim_p t$ if and only if $f(s)=f(t)$ for any minimal trace map $f$ on the corresponding semigroup ring (Proposition~\ref{tr-prim}).

\section{Conjugacy Definitions and Basics} \label{def-sect}

In this section we define the various conjugacy relations on semigroups that will be used, and explain other bits of commonly occurring notation, for convenience of reference.

We denote the set of all integers by $\Z$, the set of positive integers by $\Z^+$, and the set of natural numbers (including $0$) by $\N$. For a set $\Omega$, we denote the cardinality of $\Omega$ by $|\Omega|$. Given a semigroup $S$ we denote by $S^1$ the monoid resulting from adjoining an identity element $1$ to $S$. If $S$ is itself a monoid, we understand $S^1$ to refer to $S$.

\begin{definition}\label{sim-def1}
Let $S$ be a semigroup, and $s,t \in S$.  Write $s \sim_p^1 t$ if there exist $p, r \in S^1$ such that 
\[s=pr, \ rp=t.\]
Let $\sim_p$ denote the transitive closure of the relation $\sim_p^1$. That is, $s \sim_p t$ if there exist $p_1, r_1, p_2, r_2, \dots, p_n, r_n \in S^1$ such that \[s=p_1r_1, \ r_1p_1=p_2r_2, \ r_2p_2=p_3r_3, \dots, \ r_{n-1}p_{n-1}=p_nr_n, \ r_np_n=t.\]
The relation $\sim_p^1$, and, by extension $\sim_p$, is called the \emph{primary conjugacy} relation.
\end{definition}

We should note that while many different notations have been used for the primary conjugacy in the literature, in~\cite{ABG,ABKKMM,AKKM1,AKKM2,AKM,K} the symbol $\sim_p$ is used to denote what we are calling $\sim_p^1$ above, while $\sim_p^*$ is used for the transitive version. We have chosen our scheme, however, since the latter relation is of more central interest here, since, as far as adornment goes, ``$1$" is more descriptive than ``$*$", and since $\sim_p^n$ can be used with positive integers other than $1$ as values of $n$, to denote the number of transitions (see~\cite{AbL,AlL,LN}). 

Next we define natural variations of $\sim_p^1$ and $\sim_p$. They were initially inspired by analogous relations on rings introduced by Leroy and Nasernejad in~\cite[Definitions 3.1]{LN}. A similar, but stronger, relation is studied in~\cite{CCT}.

\begin{definition} \label{symm-conj-def}
Let $S$ be a semigroup, and $s,t \in S$. Write $s \sim_s^1 t$ if there exist $n \in \Z^+$, $p_1, \dots, p_n \in S^1$, and $f \in \B(\{1, \dots, n\})$, the symmetric group on $\, \{1, \dots, n\}$, such that \[s=p_1\cdots p_{n}, \ p_{f(1)}\cdots p_{f(n)}=t.\]
We denote by $\sim_s$ the transitive closure of the relation $\sim_s^1$. We refer to $\sim_s$ as the \emph{symmetric} or \emph{permutation} $($\emph{conjugacy}$)$ relation. 
\end{definition}

Clearly, in any semigroup, $\sim_p^1 \, \subseteq \, \sim_s^1 \, \subseteq \, \sim_s$ and $\sim_p^1 \, \subseteq \,\sim_p \, \subseteq \, \sim_s$. We show below, however, that there are semigroups where $\sim_p \, \neq \, \sim_s$, $\sim_p^1 \, \neq \, \sim_s^1$, $\sim_s^1 \, \not\subseteq \, \sim_p$ (Example~\ref{group-eg} or Proposition~\ref{free-semi-prop}), and $\sim_s^1 \, \neq \, \sim_s$, $\sim_p \, \not\subseteq \, \sim_s^1$ (Example~\ref{rees-eg}).

Let us next recall other relations that have been proposed as suitable notions of conjugacy for semigroups, which we shall compare to $\sim_s^1$ and $\sim_s$ in various parts of the paper. We mention only equivalence relations (as opposed to arbitrary relations) that apply to all semigroups (and not just special classes of them, such as inverse semigroups and epigroups). See~\cite{ABKKMM,AKKM1} for overviews of the history and literature pertaining to these and other, more specialized, semigroup conjugacy relations, along with comparisons of their properties.

\begin{definition} \label{sim-def2}
Let $S$ be a semigroup, and $s,t \in S$. 

Write $s \sim_o t$ if there exist $p,r \in S^1$ such that 
\[sp = pt, \ rs=tr. \]

Write $s \sim_n t$ if there exist $p,r \in S^1$ such that 
\[sp=pt, \  rs=tr, \ rsp=t, \ ptr=s.\]

Write $s \sim_w t$ if there exist $p,r \in S^1$ and $m \in \Z^+$ such that 
\[sp=pt, \  rs=tr, \ pr=s^m, \ rp=t^m.\]

Write $s \sim_c t$ if there exist $p \in \mathbb{P}(s)$ and $r \in \mathbb{P}(t)$ such that 
\[sp = pt, \ rs=tr,\]
where for each $s \in S \setminus \{0\}$, $\mathbb{P}(s) = \{p \in S^1 \mid \forall r \in S^1 \ (rs \neq 0 \implies rsp \neq 0)\}$.
\end{definition}

The relations on semigroups given in Definitions~\ref{sim-def1} and~\ref{sim-def2} mostly arose from attempts to translate the equation defining conjugacy in groups, or \emph{group-conjugacy}, namely $s = ptp^{-1}$, to semigroups, and then possibly compensate for any resulting deficiencies. (Specifically, $\sim_p^1$ is generally not transitive, but $\sim_p$ is; $\sim_o$ is universal in any semigroup with zero, but $\sim_c$ is generally not; $\sim_n$ is a stronger version of $\sim_o$ that reduces to $s = ptp^{-1}$ and $t=p^{-1}sp$ in any inverse semigroup; $\sim_w$ is a weaker version of $\sim_p$ that was first defined on certain matrices, in the context of symbolic dynamics, where it is known as \emph{shift equivalence}~\cite[Section 7.3]{LM}.) In contrast to this approach, as we shall see in Corollary~\ref{img-cor}, $\sim_s$ can be viewed as translating to semigroups certain functional, rather than equational, aspects of group-conjugacy.

Here is a summary of the relationships between the relations in Definitions~\ref{sim-def1} and~\ref{sim-def2}.

\begin{proposition}[Proposition 2.3 in~\cite{K}, Section 1 in~\cite{AKKM1}, Proposition 3.1 in~\cite{ABKKMM}] \label{conj-compar}
In any semigroup, $\sim_n \, \subseteq \, \sim_p \, \subseteq \, \sim_w \, \subseteq \, \sim_o$ and $\sim_n \, \subseteq \, \sim_c \, \subseteq \, \sim_o$, but $\sim_p$ and $\sim_c$ may not be comparable. Moreover, there are semigroups where $\sim_c \, \neq \, \sim_o$.
\end{proposition}

We note that the proof of~\cite[Proposition 2.3]{K} actually shows that $\sim_n \, \subseteq \, \sim_p^1$ (this follows quickly from the definitions). We give a complete description of how the various relations defined above interact with each other in Section~\ref{rel-comp-sect}.

\section{Symmetric Relation and Commutative Congruences} \label{sym-sect}

We begin by describing the precise relationship between $\sim_p$ and $\sim_s$, and characterizing the latter. The following result is straight-forward, but will be fundamental to everything that follows.

Recall that given a semigroup $S$, an equivalence relation $\rho \subseteq S \times S$ is a \emph{congruence} if $s \rho t$ implies that $(sr) \rho (tr)$ and $(rs) \rho (rt)$ for all $r,s,t \in S$. 

\begin{theorem} \label{least-comm}
Let $S$ be a semigroup, and let $\, \approx$ denote any of $\sim_p^1$, $\sim_p$, $\sim_s^1$. Then $\sim_s$ is the congruence generated by $\, \approx$, and it is the least congruence $\rho$ on $S$ such that $S/\rho$ is commutative.
\end{theorem}

\begin{proof}
Suppose that $s \sim_s^1 t$ for some $s,t \in S$, and write $s=p_1\cdots p_{n}$, $t=p_{f(1)}\cdots p_{f(n)}$ for some $n \in \Z^+$, $p_1, \dots, p_n \in S^1$, and $f \in \B(\{1, \dots, n\})$. Also let $r=p_{n+1} \in S$. Then $sr = p_1 \cdots p_np_{n+1}$ and $tr = p_{f(1)}\cdots p_{f(n)}p_{n+1}$, which implies that $sr \sim_s^1 tr$. 

Now suppose that $s \sim_s t$ for some $s,t \in S$. Then there exist $q_1, \dots, q_m \in S$ such that 
\[s = q_1 \sim_s^1 q_2 \sim_s^1 \cdots \sim_s^1 q_m = t.\]
From the previous paragraph it follows that $sr \sim_s tr$ for all $r \in S$. Analogously, if $s \sim_s t$, then $rs \sim_s rt$ for all $r \in S$. Since $\sim_s$ is clearly an equivalence relation, we conclude that it is a congruence.

Let $\rho$ be any congruence on $S$ such that $\sim_p^1 \, \subseteq \rho$, and let us denote the $\rho$-congruence class of each $s \in S$ by $[s]_{\rho}$. Then for all $s,t \in S$ we have $st \sim_p^1 ts$, and hence 
\[[s]_{\rho}[t]_{\rho} = [st]_{\rho} = [ts]_{\rho} = [t]_{\rho}[s]_{\rho}\]
(see, e.g.,~\cite[Theorem 1.5.2]{H}). Therefore the quotient semigroup $S/\rho$ is commutative. In particular, $S/\sim_s$ must be commutative, since $\sim_p^1 \, \subseteq \, \sim_s$.

Next suppose that $\rho$ is a congruence on $S$ such that $S/\rho$ is commutative, and again denote by $[s]_{\rho}$ the $\rho$-congruence class of $s \in S$. Then for all $n \in \Z^+$, $p_1, \dots, p_n \in S^1$, and $f \in \B(\{1, \dots, n\})$ we have 
\[[p_1 \cdots p_n]_{\rho} = [p_1]_{\rho} \cdots [p_n]_{\rho} = [p_{f(1)}]_{\rho}\cdots [p_{f(n)}]_{\rho} = [p_{f(1)}\cdots p_{f(n)}]_{\rho};\] 
i.e., $(p_1 \cdots p_n) \rho (p_{f(1)}\cdots p_{f(n)})$. Therefore $\sim_s^1 \, \subseteq \rho$, and since $\rho$ is transitive, it follows that $\sim_s \, \subseteq \rho$. Hence $\sim_s$ is the least congruence on $S$ that produces a commutative quotient semigroup. 

Finally, let $\rho_1$, $\rho_2$, and $\rho_3$ denote the congruences on $S$ generated by $\sim_p^1$, $\sim_p$, and $\sim_s^1$, respectively. Since $\sim_p^1 \, \subseteq \rho_i$, an earlier computation shows that $S/\rho_i$ is commutative, for each $i$. Therefore, by the previous paragraph, $\sim_s \, \subseteq \rho_i$ for each $i$. But since $\sim_p^1, \sim_p, \sim_s^1 \, \subseteq \, \sim_s$, and $\sim_s$ is a congruence, we conclude that $\sim_s \, = \rho_1 = \rho_2 = \rho_3$.
\end{proof}

We shall show in Proposition~\ref{free-small-gen}, that a relation smaller than $\sim_p^1$ typically does not generate a commutative congruence. So, in particular, we could not have included $\sim_n$ in the list of possible values of $\approx$ in the previous result (see Proposition~\ref{conj-compar}). We shall also show, in Example~\ref{incomp-eg}, that $\sim_s$ is generally not comparable to $\sim_o$, $\sim_w$, and $\sim_c$.

See~\cite[Proposition 4.2]{CCT} for a characterization of the least congruence that results in a cancellative commutative semigroup, and~\cite[Theorem 2.6]{P} for a characterization of the least commutative congruence on an inverse semigroup. Both of these relations are somewhat cumbersome to describe, and so we shall not do that here.

The next corollary is a restatement of Theorem~\ref{least-comm} in the context of homomorphisms. It also shows that $\sim_s$ performs (to a more complete extent than $\sim_p$) a certain function of the usual conjugacy in groups, namely relating the elements that must be conflated by any homomorphism with a commutative image.

\begin{corollary} \label{img-cor}
The following are equivalent for any homomorphism $f: S \to T$ of semigroups.
\begin{enumerate}
\item[$(1)$] The image $f(S)$ of $f$ in $T$ is commutative.
\item[$(2)$] For all $s,t \in S$, $s \approx t$ implies that $f(s) = f(t)$, where $\, \approx$ is any of $\sim_p^1$, $\sim_p$, $\sim_s^1$, $\sim_s$.
\end{enumerate}
If $S$ and $T$ are groups, then these are also equivalent to the following.
\begin{enumerate}
\item[$(3)$] For all group-conjugate $s,t \in S$ we have $f(s) = f(t)$.
\item[$(4)$] $[S,S] \subseteq \ker (f)$, where $\, [S,S]$ is the subgroup of $S$ generated by its multiplicative commutators $sts^{-1}t^{-1}$.
\end{enumerate}
\end{corollary}

\begin{proof}
By Theorem~\ref{least-comm}, (1) is equivalent to $\sim_s$ being contained in the kernel of $f$. (See~\cite[Theorem 1.5.2]{H} for more details.) Since the kernel of $f$ is a congruence on $S$, for any relation $\approx$ on $S$, being contained in the kernel of $f$ is equivalent to the congruence generated by $\approx$ being contained in the kernel of $f$. Hence, again by Theorem~\ref{least-comm}, (1) is equivalent to $\approx$ being contained in the kernel of $f$, where $\approx \ \in \{\sim_p^1, \sim_p,\sim_s^1, \sim_s\}$, which is precisely what (2) says.

Let us now assume that $S$ and $T$ are groups. Then (3) is a special case of (2), since group-conjugacy coincides with $\sim_p^1$ in any group. Next, for all $s,t \in S$ we have
\[f(sts^{-1}) = f(t) \iff f(sts^{-1})f(t)^{-1} = 1 \iff f(sts^{-1}t^{-1}) = 1,\]
from which the equivalence of (3) and (4) follows. Finally, (4) implies (1), since if $1 = f(sts^{-1}t^{-1})$ for all $s,t \in S$, then $1 = f(s)f(t)f(s)^{-1}f(t)^{-1}$, and so $f(s)f(t) = f(t)f(s)$.
\end{proof}

Statement (1) in the next corollary generalizes~\cite[Theorem 5.4]{AKKM1}, which shows that $\sim_p^1$ is the identity relation if and only if $S$ is commutative.

\begin{corollary}\label{univ-rel-cor}
The following hold for any semigroup $S$.
\begin{enumerate}
\item[$(1)$] The semigroup $S$ is commutative if and only if $\, \approx$ is the identity relation, where $\, \approx$ is any of $\sim_p^1$, $\sim_p$, $\sim_s^1$, $\sim_s$.
\item[$(2)$] The semigroup $S$ has no nontrivial commutative homomorphic images if and only if $\sim_s$ is the universal relation on $S$.
\end{enumerate}
\end{corollary}

\begin{proof}
(1) This follows immediately from the equivalence of (1) and (2) in Corollary~\ref{img-cor}, upon taking $T=S$ and letting $f: S \to S$ be the identity homomorphism.

(2) The relation $\sim_s$ is universal if and only if $S/\sim_s \, \cong \{0\}$ if and only if $S$ has no nontrivial commutative homomorphic images, by Theorem~\ref{least-comm}.
\end{proof}

Clearly, if $\sim_p$ (or $\sim_p^1$, or $\sim_s^1$) is the universal relation on a semigroup, then so is $\sim_s$. The converse need not hold, however. For example, $\sim_s$ is universal in any Rees matrix semigroup with a sandwich matrix having $0$ entries, according to Corollary~\ref{rees-0-sym} below. However, $\sim_p$ is not universal in such a semigroup, provided that the sandwich matrix has any nonzero entries, by Theorem~\ref{rees-prim}.

Next we strengthen an observation made in Corollary~\ref{img-cor}, and show that $\sim_s$ and $\sim_s^1$ result in a natural relation on any group $G$, namely being in the same coset of the commutator subgroup $[G,G]$. We note that the relation introduced in~\cite{CCT} also reduces to membership in the commutator subgroup--see~\cite[Theorem 2.1]{CO}.

\begin{corollary} \label{group-sym}
Let $G$ be a group, and $s,t \in G$. Then $st^{-1} \in [G,G]$ if and only if $s \sim_s^1 t$ if and only if $s \sim_s t$.
\end{corollary}

\begin{proof}
Suppose that $s \sim_s t$, let $T = G/[G,G]$, and let $f : G \to T$ be the natural projection. Then $\ker(f) = [G,G]$, and so $f(s) = f(t)$, by Corollary~\ref{img-cor}. Thus $st^{-1} \in [G,G]$.

Next suppose that $st^{-1} \in [G,G]$, and write 
\[st^{-1} = p_1r_1p_1^{-1}r_1^{-1} \cdots p_nr_np_n^{-1}r_n^{-1}\] 
for some $p_i,r_i \in G$. Then 
\[s = (p_1r_1p_1^{-1}r_1^{-1} \cdots p_nr_np_n^{-1}r_n^{-1})t, \text{ and } t = (p_1p_1^{-1})(r_1r_1^{-1}) \cdots (p_np_n^{-1})(r_nr_n^{-1})t,\]
showing that $s \sim_s^1 t$. Finally, $s \sim_s^1 t$ certainly implies that $s \sim_s t$.
\end{proof}

We note that while group-conjugacy is contained in $\sim_s$, since $\sim_p$ is, this relation is generally larger in a group, as the next example shows.

\begin{example} \label{group-eg}
Let $\Omega$ be any finite set, and let $\B (\Omega)$ be the group of all permutations of $\Omega$. According to~\cite[Theorem 1]{O}, $[\B (\Omega), \B (\Omega)]$ is the alternating subgroup of $\B (\Omega)$. So, by Corollary~\ref{group-sym}, for any $s,t \in \B(\Omega)$, we have $s \sim_s t$ (and $s \sim_s^1 t$) if and only if $st^{-1}$ is an even permutation. On the other hand, it is well-known, and easy to see, that two elements of $\B (\Omega)$ are group-conjugate if and only if they have the same number of orbits of each size. It follows that $\sim_s \, = \, \sim_s^1$ is strictly larger than group-conjugacy in $\B (\Omega)$.

More concretely, let $\Omega=\{1, 2, 3, 4, 5, 6\}$, and let $s = (12)(34)(56)$ and $t = (56)$ be elements of $\B (\Omega)$, written in cycle notation. Then $s$ and $t$ are not group-conjugate, since $s$ has three nontrivial orbits, whereas $t$ has only one. However, $st^{-1} = (12)(34)$ is an element of $[\B (\Omega), \B (\Omega)],$ and so $s \sim_s t$. 
\end{example}

Next we give an analogue of Corollary~\ref{group-sym} for rings, as well as an analogue of~\cite[Theorem 3.15(2)]{AlL}, which says that if $s \sim_p t$, for a pair of elements $s,t$ in a ring, then $s-t$ is a sum of additive commutators.

\begin{corollary} \label{comm-ideal}
Let $R$ be a $($not necessarily unital$)$ ring, and let $\, [R,R]$ denote the ideal of $R$ generated by its additive commutators $pr-rp$. Then $\, [R,R]$ is the additive subgroup of $R$ generated by elements of the form $s-t$, where $s,t \in R$ and $s \sim_s^1 t$ $($or $s \sim_s t$$)$.
\end{corollary}

\begin{proof}
Clearly, $R/[R,R]$ is a commutative ring, and hence a commutative semigroup. Thus, by Theorem~\ref{least-comm}, if $s \sim_s t$, for some $s,t \in R$, then $s-t \in [R, R]$. 

Next let $I_1$, respectively $I_2$, denote the additive subgroup of $R$ generated by elements of the form $s-t$ ($s,t \in R$), where $s \sim_s^1 t$, respectively $s \sim_s t$. Then, by the previous paragraph, $I_1 \subseteq I_2 \subseteq [R,R]$. Now, as an additive group, $[R,R]$ is generated by elements of the form 
\[q(rs-sr)t = qrst-qsrt\]
($q,r,s,t \in R$). Since $qrst \sim_s^1 qsrt$, we see that $[R,R] \subseteq I_1$, and hence $[R,R] = I_1 = I_2$.
\end{proof}

The next example shows that membership in the same coset of the commutator ideal of a ring is generally strictly larger than $\sim_s$, in contrast to the situation with the commutator subgroup of a group.

\begin{example}
Let $F$ be a field, $n \geq 2$, and $R = \M_n(F)$ the ring of $n\times n$ matrices over $F$. Also let $s \in R$ be any invertible matrix with trace $0$. Since $s$ has trace $0$, it is an additive commutator in $R$, by the Shoda--Albert--Muckenhoupt theorem~\cite{AM}, and so $s=s-0 \in [R,R]$. However, since $s$ has a nonzero determinant, if $s=p_1\cdots p_n$ for some $p_1, \dots, p_n \in R$, then each $p_i$ must also have a nonzero determinant. Thus, any $t \in R$, such that $s \sim_s^1 t$, must have the same property. From this it follows that the $\sim_s$-equivalence class of $s$ in $R$ contains only invertible matrices, and, in particular, $s \not\sim_s 0$.
\end{example}

\section{Symmetric Relation Basics}

In this section we explore basic properties of $\sim_s^1$ and $\sim_s$, which will be used extensively in what follows. Some of these properties are unique to $\sim_s$, among the various relations mentioned in Definitions~\ref{sim-def1},~\ref{symm-conj-def},~\ref{sim-def2}.

Statement (2) in the next lemma is a convenient reformulation of the claim in Theorem~\ref{least-comm} that $\sim_s$ is a congruence, whereas (3) and (4) show that $\sim_s^1$ and $\sim_s$ share a standard property of group-conjugacy. Statements (1) and (4) are based on results of Leroy and Nasernejad for rings, and can be proved the same way. But since the arguments are short, we give them here, for convenience. 

\begin{lemma} \label{Leroy-lemma2}
Let $S$ be a semigroup, and $s_1,s_2,t_1,t_2 \in S$. 
\begin{enumerate}
\item[$(1)$] $($cf.\ Lemma 3.2(iii) in~\cite{LN}.$)$ If $s_1 \sim_s^1 t_1$ and $s_2 \sim_s^1 t_2$, then $s_1s_2 \sim_s^1 t_1t_2$. 
\item[$(2)$] If $s_1 \sim_s t_1$ and $s_2 \sim_s t_2$, then $s_1s_2 \sim_s t_1t_2$. 
\item[$(3)$] If $s_1 \sim_s^1 t_1$, then $s_1^n \sim_s^1 t_1^n$ for all $n \in \Z^+$. 
\item[$(4)$] $($cf.\ Theorem 4.3(i) in~\cite{LN}.$)$ If $s_1 \sim_s t_1$, then $s_1^n \sim_s t_1^n$ for all $n \in \Z^+$. 
\end{enumerate}
\end{lemma}

\begin{proof}
(1) Suppose that $s_1 \sim_s^1 t_1$ and $s_2 \sim_s^1 t_2$. Then there exist $n,m \in \Z^+$, with $n<m$, as well as $p_1, \dots, p_{m} \in S^1$, $f_1 \in \B(\{1, \dots, n\})$, and $f_2 \in \B(\{n+1, \dots, m\})$ such that $s_1=p_1\cdots p_{n}$, $s_2=p_{n+1}\cdots p_{m}$, $t_1 = p_{f_1(1)}\cdots p_{f_1(n)}$, and $t_2 = p_{f_2(n+1)}\cdots p_{f_2(m)}$. Let $g \in  \B(\{1, \dots, m\})$ be such that $g$ agrees with $f_1$ on $\{1, \dots, n\}$, and agrees with $f_2$ on $\{n+1, \dots, m\}$. Then 
\[s_1s_2 =  p_1\cdots p_{n}p_{n+1}\cdots p_{m} \sim_s^1 p_{g(1)}\cdots p_{g(n)}p_{g(n+1)}\cdots p_{g(m)}  = t_1t_2,\]
and so $s_1s_2 \sim_s^1 t_1t_2$.

(2) It is a standard fact that an equivalence relation $\rho$ on $S$ is a congruence (as defined in Section~\ref{sym-sect}) if and only if $s_1 \rho t_1$ and $s_2 \rho t_2$ imply that $(s_1s_2) \rho (t_1t_2)$ for all $s_1,s_2,t_1,t_2 \in S$ (see, e.g., \cite[Proposition 1.5.1]{H}). Thus the claim follows from Theorem~\ref{least-comm}.

(3) This follows from (1), by induction on $n$.

(4) This follows from (2), by induction on $n$.
\end{proof}

The next lemma shows that $\sim_s^1$ and $\sim_s$ interact well with some of the standard structure in an \emph{inverse semigroup}, i.e., a semigroup $S$ where for each $s \in S$ there is a unique element $s^{-1} \in S$ satisfying $s = ss^{-1}s$ and $s^{-1} = s^{-1}ss^{-1}$. For an inverse semigroup $S$, \emph{the natural partial order} $\leq$ on $S$ is defined by $s\leq t$ ($s,t \in S$) if $s=te$ for some $e \in E(S)$, the set of idempotents of $S$. Equivalently, $s\leq t$ if $s=et$ for some $e \in E(S)$. (See~\cite[\S 5.2]{H} for more details.)

\begin{lemma} \label{inv-semi-lem}
Let $S$ be an inverse semigroup, and $s,t \in S$. 
\begin{enumerate}
\item[$(1)$] If $s\sim_s^1 t$, then $s^{-1} \sim_s^1 t^{-1}$.
\item[$(2)$] If $s\sim_s t$, then $s^{-1} \sim_s t^{-1}$.
\item[$(3)$] If $s \leq t$, then for all $t' \in S$ such that $t \sim_s^1 t'$, there exists $s' \in S$ such that $s' \leq t'$ and $s \sim_s^1 s'$.
\item[$(4)$] If $s \leq t$, then for all $t' \in S$ such that $t \sim_s t'$, there exists $s' \in S$ such that $s' \leq t'$ and $s \sim_s s'$.
\end{enumerate}
\end{lemma}

\begin{proof}
(1) Suppose that $s\sim_s^1 t$. Then $s=p_1\cdots p_{n}$ and $t = p_{f(1)}\cdots p_{f(n)}$ for some $n \in \Z^+$, $p_1, \dots, p_n \in S^1$, and $f \in \B(\{1, \dots, n\})$. Hence 
\[s^{-1} = p_n^{-1}\cdots p_1^{-1} \sim_s^1 p_{f(n)}^{-1}\cdots p_{f(1)}^{-1} = t^{-1}\]
(see, e.g., \cite[Proposition 5.1.2(1)]{H}).

(2) This is a consequence of $\sim_s$ being a congruence, but can also be shown directly. Specifically, suppose that $s\sim_s t$. Then there exist $r_1, \dots, r_n \in S$ such that 
\[s = r_1 \sim_s^1 r_2 \sim_s^1 \cdots \sim_s^1 r_n = t.\]
Hence $s^{-1} \sim_s t^{-1}$, by (1).

(3)  Assuming that $s \leq t$, we have $s=te$ for some $e \in E(S)$. Suppose that $t' \in S$ is such that $t \sim_s^1 t'$, and let $s' = t'e$. Then $s' \leq t'$, and $s = te \sim_s^1 t'e = s'$, by Lemma~\ref{Leroy-lemma2}(1).

(4) This can be shown by the same argument as (3), but using Lemma~\ref{Leroy-lemma2}(2).
\end{proof}

Then next observation gives a succinct alternative description of $\sim_s$.

\begin{proposition}[cf.\ Lemma 3.2(i) in~\cite{LN}] \label{Leroy-lemma}
Let $S$ be a semigroup, and $s,t \in S$. Write $s \sim_*^1 t$ if there exist $p_1, p_2, p_3 \in S^1$ such that 
\[s=p_1p_2p_3, \ p_1p_3p_2=t,\]
and denote by $\sim_*$ the transitive closure of the relation $\sim_*^1$. Then $\sim_* \, = \, \sim_s$.
\end{proposition}

\begin{proof}
This is shown for rings in~\cite[Lemma 3.2(i)]{LN}, but the proof does not use addition, and caries over word-for-word to semigroups. However, let us outline the argument here, for the convenience of the reader.

Since, clearly, $\sim_*^1 \, \subseteq \, \sim_s$, and the latter is transitive, we have $\sim_* \, \subseteq \, \sim_s$. For the opposite inclusion, it similarly suffices to show that $\sim_s^1 \, \subseteq \, \sim_*$. Moreover, given that every permutation of a finite set is a product of transpositions, it is enough to show that if $s,t \in S$, $n \in \Z^+$, $p_1, \dots, p_n \in S^1$, and $f \in \B(\{1, \dots, n\})$ are such that $f$ is a transposition, $s=p_1\cdots p_{n}$, and $t=p_{f(1)}\cdots p_{f(n)}$, then $s \sim_* t$. Finally, writing $f = (ij)$ in cycle notation, for some $1 \leq i < j \leq n$, one shows directly that a sequence of $\sim_*^1$-transitions takes $p_1\cdots p_{n}$ to $p_{f(1)}\cdots p_{f(n)}$.
\end{proof}

We conclude this section with an examination of the closures under $\sim_s$ of various substructures of a semigroup. Statements (1) and (2) in the next proposition are generalizations of~\cite[Lemma 3.5(iii)]{LN} and~\cite[Proposition 3.10(vi)]{LN}, respectively, which pertain to rings. They, along with statement (4), are essentially consequences of $\sim_s$ being a congruence.

\begin{proposition} \label{sym-closure}
Let $S$ be a semigroup, $T \subseteq S$, and $\overline{T} \subseteq S$ the closure of $T$ under $\sim_s$. Then the following hold.
\begin{enumerate}
\item[$(1)$] $\overline{T} = \overline{\overline{T}}$.
\item[$(2)$] If $T$ is a subsemigroup of $S$, then so is $\overline{T}$.
\item[$(3)$] If $T$ is a left, respectively right, respectively two-sided, ideal of $S$, then $\overline{T}$ is a two-sided ideal.
\item[$(4)$] If $S$ is an inverse semigroup, and $T$ is an inverse subsemigroup of $S$, then so is $\overline{T}$.
\item[$(5)$] If $S$ is a group, and $T$ is a subgroup of $S$, then $\overline{T}$ is the $\,($normal$)$ subgroup of $S$ generated by $T$ and $\, [S,S]$.
\end{enumerate}
\end{proposition}

\begin{proof}
(1) Clearly $\overline{T} \subseteq \overline{\overline{T}}$. The reverse inclusion follows from the transitivity of $\sim_s$.

(2) Suppose that $T$ is a subsemigroup of $S$, and let $s,t \in \overline{T}$. Then $s \sim_s s'$ and $t \sim_s t'$ for some $s',t' \in T$. Hence $st \sim_s s't' \in T$, by Lemma~\ref{Leroy-lemma2}(2), and so $st \in \overline{T}$. Thus $\overline{T}$ is a subsemigroup.

(3) Suppose that $T$ is a left ideal, and let $s \in S$ and $t \in \overline{T}$. Also let $t' \in T$ be such that $t \sim_s t'$. Then, by Theorem~\ref{least-comm}, $st \sim_s st' \in T$, and hence $st \in \overline{T}$. Since $ts \sim_s st$, we also have $ts \in \overline{T}$, by (1). Thus $\overline{T}$ is a two-sided ideal. 

The right ideal version is entirely analogous, and the two-sided ideal version follows immediately from the one-sided ones.

(4) Suppose that $S$ is an inverse semigroup, and $T$ is an inverse subsemigroup of $S$. By (2), $\overline{T}$, is a subsemigroup of $S$. Now let $s \in \overline{T}$. Then $s \sim_s t$ for some $t \in T$. Hence $s^{-1} \sim_s t^{-1} \in T$, by Lemma~\ref{inv-semi-lem}(2), and so $s^{-1} \in \overline{T}$. Thus $\overline{T}$ is an inverse subsemigroup.

(5) Suppose that $S$ is a group, and $T$ is a subgroup of $S$. Then $\overline{T}$ is an inverse subsemigroup of $S$, by (4), and hence a subgroup. As a subgroup, $\overline{T}$ contains $1$, and hence also the commutator subgroup $[S,S]$ of $S$, since $s \sim_s 1$ for all $s \in [S,S]$. On the other hand, any subgroup of $S$ that contains $T$ and $[S,S]$ must contain $[S,S]T$, and hence also $\overline{T}$, by Corollary~\ref{group-sym}. Thus $\overline{T}$ is precisely the subgroup of $S$ generated by $T$ and $[S,S]$. Finally, any subgroup of $S$ containing $[S,S]$ is necessarily normal (see, e.g.,~\cite[\S 5.4, Proposition 7]{DF}).
\end{proof}

The relations in Definitions~\ref{sim-def1} and~\ref{sim-def2} do not generally have the properties above, other than (1), since each reduces to group-conjugacy in any group, and this relation does not preserve sub(semi)groups or one-sided ideals, as the next example shows.

\begin{example}
Let $\Omega$ be a set of cardinality at least $2$, $G_{\Omega}$ the free group on $\Omega$, and $\alpha \in \Omega$. Also let $\approx$ denote any of the relations from Definitions~\ref{sim-def1} and~\ref{sim-def2}. It is easy to see that each of those reduces to group-conjugacy in any group, and so $\approx$ is simply group-conjugacy in $G_{\Omega}$.

Let $H = \langle \alpha \rangle$ be the subgroup of $G_{\Omega}$ generated by $\alpha$, and let $\overline{H}$ denote the $\approx$-closure of $H$. Then
\[\overline{H} = \{s\alpha^ns^{-1} \mid n \in \Z, s \in G_{\Omega}\}.\]
Now take $\beta \in \Omega \setminus \{\alpha\}$. Then $\beta\alpha\beta^{-1}, \beta^2\alpha\beta^{-2} \in \overline{H}$, but
\[\beta\alpha\beta^{-1}\beta^2\alpha\beta^{-2} = \beta\alpha\beta\alpha\beta^{-2} \notin \overline{H}.\]
So $\overline{H}$ is not a subsemigroup (or inverse subsemigroup or subgroup) of $G_{\Omega}$.

Next let $H = \{s\alpha \mid s \in G_{\Omega}\}$ be the left ideal generated by $\alpha$, and let $\overline{H}$ denote the $\approx$-closure of $H$. Then 
\[\overline{H} = \{ts\alpha t^{-1} \mid n \in \Z, s, t \in G_{\Omega}\}.\]
The same argument as before shows that $\overline{H}$ is not a subsemigroup, and hence is not left ideal, of $G_{\Omega}$.

Similar considerations show that the closure of a right ideal of $G_{\Omega}$ under $\approx$ is generally not itself a right ideal.
\end{example}

In contrast to the case of one-sided ideals, it is easy to see that every ideal in any semigroup is closed under $\sim_n$ (see~\cite{K}). On the other hand, it is not hard to show that in a noncommutative free semigroup the closure of an ideal under 
$\sim^1_p \, = \, \sim_p \, = \, \sim_w \, = \, \sim_o \, = \, \sim_c$ (see Section~\ref{free-sect} for more details) is generally not an ideal.

\section{Free Semigroups} \label{free-sect}

With the exception of the appendix on semigroup rings, the remainder of this paper is primarily devoted to classifying $\sim_s$, and also $\sim_p$, in cases where it has not been completely described previously, (as well as $\sim_s^1$ and $\sim_p^1$, when that is convenient) in various classes of semigroups. Our main purpose is to compare $\sim_s$ to the different relations defined in Section~\ref{def-sect}, exhibit various properties of $\sim_s$, and to demonstrate methods for computing $\sim_s$, using its relationship to $\sim_p$ and the fact that it is a congruence. We begin with free semigroups.

It is shown in~\cite[Theorem 2.2]{AKM} that $\sim^1_p \, = \, \sim_c \, = \, \sim_o$ in any free semigroup, and hence $\sim_w$ agrees with those relations as well, by Proposition~\ref{conj-compar}. It is also easy to see that $\sim_n$ is the identity relation in any free semigroup.

By Corollary~\ref{img-cor}, if $S$ is a semigroup, $T$ is a commutative semigroup, and $f: S \to T$ is a homomorphism, then $s \sim_s t$ implies that $f(s)=f(t)$, for all $s,t \in S$. Statement (2) in the next proposition can be interpreted to mean that $\sim_s$ is the largest equivalence relation with this property, that is definable on all semigroups.

\begin{proposition} \label{free-semi-prop}
Let $\, \Omega$ be a nonempty set, $F_{\Omega}$ the free semigroup on $\, \Omega$, and $C_{\Omega}$ the free commutative semigroup on $\, \Omega$. Then the following hold.
\begin{enumerate}
\item[$(1)$] In $F_{\Omega}$, $\sim_p^1 \, = \, \sim_p \, \subseteq \, \sim_s$, and the inclusion is strict if and only if $\, |\Omega| \geq 2$.
\item[$(2)$] Let $f : F_{\Omega} \to C_{\Omega}$ the semigroup homomorphism induced by letting $f(\alpha)=\alpha$ for each $\alpha \in \Omega$. Then $f(s)=f(t)$ if and only if $s \sim_s t$ if and only if $s \sim_s^1 t$, for all $s,t \in F_{\Omega}$.
\end{enumerate}
\end{proposition}

\begin{proof}
(1) Clearly, $\sim_p^1 \,  \subseteq \, \sim_p \, \subseteq \, \sim_s$. It is shown in~\cite[Theorem 2.2]{AKM} that $\sim^1_p \, = \, \sim_o$ in $F_{\Omega}$. Since $\sim^1_p \, \subseteq \, \sim_p \, \subseteq \, \sim_o$ in any semigroup, by Proposition~\ref{conj-compar}, it follows that $\sim_p^1 \, = \, \sim_p$ in $F_{\Omega}$. Now, if $|\Omega| = 1$, and $F_{\Omega}$ is therefore commutative, then each of $\sim_p^1$, $\sim_p$, and $\sim_s$ is the identity relation (see Corollary~\ref{univ-rel-cor}(1)). In particular, $\sim_p^1 \, = \, \sim_p \, = \, \sim_s$.

Next suppose that $|\Omega| \geq 2$, and let $\alpha, \beta \in \Omega$ be distinct. Then $\alpha^2\beta^2 \not\sim_p^1 \alpha \beta \alpha \beta$, while $\alpha^2\beta^2 \sim_s \alpha \beta \alpha \beta$. Therefore $\sim_p^1 \, = \, \sim_p \, \subsetneq \, \sim_s$ in this case.

(2) Let $n \in \Z^+$, $p_1, \dots, p_n \in F_{\Omega}^1$, and $g \in \B(\{1, \dots, n\})$. Then, clearly, $f(p_1\cdots p_{n}) = f(p_{g(1)}\cdots p_{g(n)})$. It follows that if $s \sim_s^1 t$ or $s \sim_s t$, for some $s,t \in F_{\Omega}$, then $f(s) = f(t)$.

Conversely, suppose that $f(s) = f(t)$ for some $s,t \in F_{\Omega}$. Write $s = \alpha_1\cdots \alpha_n$ for some $\alpha_1, \dots, \alpha_n \in \Omega$. Then, by the definition of $f$, we have $t = \alpha_{g(1)}\cdots \alpha_{g(n)}$ for some $g \in \B(\{1, \dots, n\})$, and so $s \sim_s^1 t$, which implies that $s \sim_s t$ also.
\end{proof}

This proposition gives another example of a semigroup where $\sim_p \, \neq \, \sim_s$, $\sim_p^1 \, \neq \, \sim_s^1$, and $\sim_s^1 \, \not\subseteq \, \sim_p$ (see Example~\ref{group-eg}).

The next result explains why a relation smaller than $\sim_p^1$ typically does not generate a commutative congruence on a semigroup.

\begin{proposition} \label{free-small-gen}
Let $\, \Omega$ be a nonempty set, $F_{\Omega}$ the free semigroup on $\, \Omega$, and $\, \approx$ a reflexive symmetric relation on $F_{\Omega}$ such that $\, \approx \, \subseteq \, \sim_s$. Then the congruence generated by $\, \approx$ is $\sim_s$ if and only if $\alpha\beta \approx \beta\alpha$ for all $\alpha, \beta \in \Omega$.
\end{proposition}

\begin{proof}
Let $\rho$ be the congruence generated by $\approx$. Then $\rho \subseteq \, \sim_s$, since $\sim_s$ is a congruence, by Theorem~\ref{least-comm}. If $|\Omega| = 1$, then $F_{\Omega}$ is commutative, and so $\sim_s$ (see Corollary~\ref{univ-rel-cor}(1)), as well as $\approx$ (given that it is reflexive), is simply equality. Thus the claim holds trivially in this case, and so we may assume that $|\Omega| \geq 2$.

Suppose that $\alpha\beta \approx \beta\alpha$ for all $\alpha, \beta \in \Omega$. Since $\Omega$ generates $F_{\Omega}$ as a semigroup, the quotient $F_{\Omega}/\rho$ must be commutative, which implies that $\sim_s \, \subseteq \rho$, by Theorem~\ref{least-comm}, and so $\sim_s \, = \rho$.

Conversely, suppose that $\sim_s \, =  \rho$, but $\alpha\beta \not\approx \beta\alpha$ for some $\alpha, \beta \in \Omega$. Since $\alpha\beta \sim_s \beta\alpha$, there must exist $p_1, \dots, p_n \in F_{\Omega}$ such that $p_1 = \alpha\beta$, $p_n = \beta\alpha$, and each $p_i$ is connected to $p_{i+1}$ via an elementary $\approx$-transition. (See~\cite[Proposition 1.5.9]{H}.) Given that $\approx$ is symmetric, this means that for each $i<n$ there exist $r_i, t_i \in F_{\Omega}^1$ and $s_i,s_i' \in  F_{\Omega}$, such that $s_i \approx s_i'$, $p_i = r_is_it_i$, and $p_{i+1} = r_is_i't_i$. Taking $i=1$, since $p_1 = \alpha\beta = r_1s_1t_1$ and $s_1 \neq 1$, the possibilities are: 
\begin{enumerate}
\item[$($a$)$] $r_1 = \alpha$, $s_1 = \beta$, $t_1 = 1$;
\item[$($b$)$] $r_1 = 1$, $s_1 = \alpha$, $t_1 = \beta$;
\item[$($c$)$] $r_1 = 1$, $s_1 = \alpha\beta$, $t_1 = 1$.
\end{enumerate}
Since $s_1 \approx s_1'$ and $\, \approx \, \subseteq \, \sim_s$, we must have $s_1' = \beta$ in case (a), and $s_1' = \alpha$ in case (b). Moreover, since the $\sim_s$-equivalence class of $\alpha\beta$ is $\{\alpha\beta, \beta\alpha\}$, by Proposition~\ref{free-semi-prop}(2), and since we have assumed that $\alpha\beta \not\approx \beta\alpha$, we conclude that $s_1' = \alpha\beta$ in case (c). Therefore in each case $s_1 = s_1'$, and hence $p_2 = r_1s_1't_1 = p_1$. Iterating this argument, we conclude that $\alpha\beta = p_1 = p_2 = \cdots = p_n$, which contradicts $p_n = \beta\alpha$. Thus if $\sim_s \, =  \rho$, then $\alpha\beta \approx \beta\alpha$ for all $\alpha, \beta \in \Omega$.
\end{proof}

\section{Rees Matrix Semigroups} \label{rees-sect}

Let $G$ be a group, $G^0 = G \cup \{0\}$ the corresponding $0$-group, $I$ and $\Lambda$ nonempty sets, and $P = (p_{\lambda i})$ a $\Lambda \times I$ matrix (called a \emph{sandwich matrix}) with entries in $G^0$, such that no row or column consists entirely of zeros. Then $(I \times G \times \Lambda) \cup \{0\}$, with multiplication given by
\[(i,s,\lambda)(j,t,\mu) = \left\{ \begin{array}{cl}
(i,sp_{\lambda j}t,\mu) & \text{if } p_{\lambda j} \neq 0\\
0 & \text{otherwise}
\end{array}\right.\]
and
\[(i,s,\lambda)0 = 0 = 0(i,s,\lambda) = 0\cdot 0,\]
is a semigroup, called a \emph{Rees matrix semigroup}, and denoted by $\mathcal{M}^0(G; I, \Lambda; P)$. According to the Rees theorem (see, e.g.,~\cite[Theorem 3.2.3]{H}), $\mathcal{M}^0(G; I, \Lambda; P)$ is \emph{completely} $0$-\emph{simple} (i.e., it is a semigroup $S$ such that $S^2 \neq \{0\}$, $S$ and $\{0\}$ are the only ideals, and the inverse semigroup $E(S)$ of idempotents of $S$ has an element minimal in the natural partial order $\leq$), and every completely $0$-simple semigroup is of this form. (See~\cite[\S 3.2]{H} for more details.) 

It is shown in~\cite[Proposition 4.26]{AKKM1} that $\sim_c \, \subseteq \, \sim_p^1$ in $\mathcal{M}^0(G; I, \Lambda; P)$, with equality if and only if $P$ has only nonzero elements. The relation $\sim_n$ in these semigroups is classified in~\cite[Theorem 2.25]{ABKKMM}. 

We begin by giving a complete characterization of $\sim_p$ and $\sim_p^1$ in $\mathcal{M}^0(G; I, \Lambda; P)$, which will then help us describe $\sim_s$. Interestingly, here $\sim_p$ and $\sim_p^1$ coincide with $\sim_n$, except $(i,s,\lambda)$ is in a $\sim_n$-equivalence class of its own whenever $p_{\lambda i} = 0$, whereas in this situation $(i,s,\lambda) \sim_p^1 0$. In particular, in Rees matrix semigroups we generally have $\sim_n \, \subsetneq \, \sim_p^1$.

\begin{theorem} \label{rees-prim}
Let $\mathcal{M}^0(G; I, \Lambda; P)$ be a Rees matrix semigroup, with appropriate $G$, $I$, $\Lambda$, and $P$, and let $\, (i,s,\lambda), (j,t,\mu) \in \mathcal{M}^0(G; I, \Lambda; P) \setminus \{0\}$. Then the following hold.
\begin{enumerate}
\item[$(1)$] We have $\, (i,s,\lambda) \sim_p^1 (j,t,\mu)$ if and only if either $\, (i,s,\lambda) = (j,t,\mu)$, or $p_{\lambda i} \neq 0 \neq p_{\mu j}$ and $rp_{\lambda i}s = tp_{\mu j}r$ for some $r \in G$. Also $\, (i,s,\lambda) \sim_p^1 0$ if and only if $p_{\lambda i} = 0$.

\item[$(2)$] We have $\, (i,s,\lambda) \sim_p (j,t,\mu)$ if and only if either $p_{\lambda i} = 0 = p_{\mu j}$, or $p_{\lambda i} \neq 0 \neq p_{\mu j}$ and $rp_{\lambda i}s = tp_{\mu j}r$ for some $r \in G$. Also $\, (i,s,\lambda) \sim_p 0$ if and only if $p_{\lambda i} = 0$.
\end{enumerate}
\end{theorem}

\begin{proof}
(1) First, suppose that $p_{\lambda i} = 0$. By the definition of the sandwich matrix, we can find $k \in I$ and $\nu \in \Lambda$ such that $p_{\nu k} \neq 0$. Then taking $r = sp_{\nu k}^{-1} \in G$, we have
\[(i,s,\lambda) = (i,rp_{\nu k},\lambda) = (i,r,\nu)(k,1,\lambda) \text{ and }(k,1,\lambda)(i,r,\nu) = 0.\] 
Therefore $(i,s,\lambda) \sim_p^1 0$.

Conversely, suppose that $(i,s,\lambda) \sim_p^1 0$. Then there exist $r_1, r_2 \in G$, $\nu \in \Lambda$, and $k \in I$ such that 
\[(i,s,\lambda) = (i,r_1,\nu)(k,r_2,\lambda) \text{ and } (k,r_2,\lambda)(i,r_1,\nu) = 0.\] 
It follows from the second equation that $p_{\lambda i} = 0$, which proves the second claim in (1).

The last computation also shows that if $p_{\lambda i} = 0$, then $(i,s,\lambda)$ and $0$ are the only elements of $\mathcal{M}^0(G; I, \Lambda; P)$ that are $\sim_p^1$-related to $(i,s,\lambda)$. Therefore if $(i,s,\lambda) \sim_p^1 (j,t,\mu)$, then either $(i,s,\lambda) = (j,t,\mu)$ or $p_{\lambda i} \neq 0 \neq p_{\mu j}$.

For the remainder of the proof of (1), let us assume that $p_{\lambda i} \neq 0 \neq p_{\mu j}$. We shall complete the argument by showing that in this case, $(i,s,\lambda) \sim_p^1 (j,t,\mu)$ if and only if $rp_{\lambda i}s = tp_{\mu j}r$ for some $r \in G$.

Given that $p_{\lambda i} \neq 0 \neq p_{\mu j}$, we have $(i,s,\lambda) \sim_p^1 (j,t,\mu)$ if and only if there exist $r_1, r_2 \in G$ such that 
\[(i,s,\lambda) = (i,r_1,\mu)(j,r_2,\lambda) = (i,r_1p_{\mu j}r_2,\lambda)\]
and
\[(j,t,\mu) = (j,r_2,\lambda)(i,r_1,\mu) = (j,r_2p_{\lambda i}r_1,\mu).\]
This is the case if and only if there exist $r_1, r_2 \in G$ such that $s = r_1p_{\mu j}r_2$ and $t = r_2p_{\lambda i}r_1$. Rearranging these equations gives $r_1 = sr_2^{-1}p_{\mu j}^{-1}$ and $r_2 = tr_1^{-1}p_{\lambda i}^{-1}$, respectively. Substituting these into $t = r_2p_{\lambda i}r_1$ and $s = r_1p_{\mu j}r_2$, respectively, gives $r_2p_{\lambda i}s = tp_{\mu j}r_2$ and $r_1p_{\mu j}t = s p_{\lambda i} r_1$. Since this computation is reversible, we conclude that $(i,s,\lambda) \sim_p^1 (j,t,\mu)$ if and only if there exist $r_1, r_2 \in G$ such that $r_2p_{\lambda i}s = tp_{\mu j}r_2$ and $r_1p_{\mu j}t = s p_{\lambda i} r_1$. It is easy to see that satisfying one of these equations implies satisfying the other, and so only one is needed. Thus $(i,s,\lambda) \sim_p^1 (j,t,\mu)$ if and only if $rp_{\lambda i}s = tp_{\mu j}r$ for some $r \in G$, as claimed.

(2) If $p_{\lambda i} = 0$, then $(i,s,\lambda) \sim_p 0$, by (1). Conversely, suppose that $(i,s,\lambda) \sim_p 0$. Then there exist $q_1, \dots, q_n \in \mathcal{M}^0(G; I, \Lambda; P)$ such that
\[(i,s,\lambda) \sim_p^1 q_1 \sim_p^1 q_2 \sim_p^1 \cdots \sim_p^1 q_n \sim_p^1 0,\]
where we may assume that each $q_i \neq 0$. By (1), $q_n \sim_p^1 0$ implies that $q_{n-1} = q_n$. It follows inductively that $(i,s,\lambda) = q_1 = \cdots = q_n$, and hence  $(i,s,\lambda) \sim_p^1 0$. Therefore $p_{\lambda i} = 0$, by (1). This proves the second claim in (2).

Next, by (1), since $\sim_p$ is transitive, if $p_{\lambda i} = 0 = p_{\mu j}$, then $(i,s,\lambda) \sim_p (j,t,\mu)$. Moreover, by the previous paragraph, it cannot be the case that $(i,s,\lambda) \sim_p (j,t,\mu)$, and exactly one of $p_{\lambda i}$ and $p_{\mu j}$ is $0$. Therefore to conclude the proof it suffices to assume that $p_{\lambda i} \neq 0 \neq p_{\mu j}$, and show that in this case, $(i,s,\lambda) \sim_p (j,t,\mu)$ if and only if $rp_{\lambda i}s = tp_{\mu j}r$ for some $r \in G$.

Given that $p_{\lambda i} \neq 0 \neq p_{\mu j}$, by (1) and~\cite[Theorem 2.25]{ABKKMM},  we have $(i,s,\lambda) \sim_p^1 (j,t,\mu)$ if and only if $(i,s,\lambda) \sim_n (j,t,\mu)$. This implies that $\sim_p^1$ is an equivalence relation on elements $(i,s,\lambda) \in \mathcal{M}^0(G; I, \Lambda; P)$ with $p_{\lambda i} \neq 0$, and therefore $\sim_p^1 \, = \, \sim_p$ in this situation. Alternatively, one can show directly that the relation $rp_{\mu j}t = s p_{\lambda i} r$ for some $r \in G$ (and $r'p_{\mu j}t = s p_{\lambda i} r'$ for some $r' \in G$), on elements $(i,s,\lambda), (j,t,\mu) \in \mathcal{M}^0(G; I, \Lambda; P)$, is transitive. So if $(i,s,\lambda) \sim_p (j,t,\mu)$, then $rp_{\lambda i}s = tp_{\mu j}r$ for some $r \in G$. Hence, in the case where $p_{\lambda i} \neq 0 \neq p_{\mu j}$, by (1), we have $(i,s,\lambda) \sim_p (j,t,\mu)$ if and only if $rp_{\lambda i}s = tp_{\mu j}r$ for some $r \in G$, as desired.
\end{proof}

Using the previous result we can construct an example of a semigroup where $\sim_p \, \not\subseteq \, \sim_s^1$, and hence also $\sim_s^1 \, \neq \, \sim_s$ (and $\sim_p^1 \, \neq \, \sim_p$).

\begin{example} \label{rees-eg}
Let $G$ be any group that is not equal to its commutator subgroup $[G,G]$ (e.g., a noncommutative free group), let $s,t \in G$ be such that $s\not\sim_s t$ in $G$ (which exist, by Corollary~\ref{group-sym}), let $I = \Lambda = \{1,2\}$, and let 
\[P =\left(\begin{array}{cc}
0 & 1 \\
1 & 0 \\
\end{array}\right).\]
Then, by Theorem~\ref{rees-prim}(1), $(1,s,1)\sim_p^1 0 \sim_p^1 (1,t,1)$, and so $(1,s,1) \sim_p (1,t,1)$, in the semigroup $\mathcal{M}^0(G; I, \Lambda; P)$.

Next suppose that $(1,s,1) \sim_s^1 (1,t,1)$. Then there exist $(a_i, p_i, b_i) \in \mathcal{M}^0(G; I, \Lambda; P)$ and $f \in \B(\{1,\dots, n\})$, where $i \in \{1, \dots, n\}$, such that
\[(1,s,1) = (a_1, p_1, b_1) \cdots (a_n, p_n, b_n) = (a_1, p_1\cdots p_n, b_n)\]
and
\[(1,t,1) = (a_{f(1)}, p_{f(1)}, b_{f(1)}) \cdots (a_{f(n)}, p_{f(n)}, b_{f(n)}) = (a_{f(1)}, p_{f(1)} \cdots p_{f(n)}, b_{f(n)})\]
(using the fact that each entry in $P$ is either $0$ or $1$). In particular, $s = p_1\cdots p_n$ and $t = p_{f(1)} \cdots p_{f(n)}$, which implies that $s \sim_s t$ in $G$, contrary to hypothesis. Thus $(1,s,1) \not\sim_s^1 (1,t,1)$.
\end{example}

Next we use Theorem~\ref{rees-prim}, along with the fact that $\sim_s$ is a congruence, to describe this relation in Rees matrix semigroups. We begin with the case where the sandwich matrix has at least one zero.

\begin{corollary} \label{rees-0-sym}
Let $\mathcal{M}^0(G; I, \Lambda; P)$ be a Rees matrix semigroup, with appropriate $G$, $I$, $\Lambda$, and $P$. If $P$ has any $\, 0$ entries, then $\sim_s$ is the universal relation on $\mathcal{M}^0(G; I, \Lambda; P)$.
\end{corollary}

\begin{proof}
Since, by Theorem~\ref{least-comm}, $\sim_s$ is a congruence, the $\sim_s$-equivalence class of $0$ is an ideal of $\mathcal{M}^0(G; I, \Lambda; P)$. If $P$ has any $0$ entries, then the $\sim_p$-equivalence class of $0$, and hence also the $\sim_s$-equivalence class of $0$, contains nonzero elements, by Theorem~\ref{rees-prim}. Since $\mathcal{M}^0(G; I, \Lambda; P)$ is completely $0$-simple, the only nonzero ideal is $\mathcal{M}^0(G; I, \Lambda; P)$. Thus the $\sim_s$-equivalence class of $0$, in this case, must be $\mathcal{M}^0(G; I, \Lambda; P)$, and so $\sim_s$ is the universal relation.
\end{proof}

If the sandwich matrix $P$ has only nonzero entries, then $\mathcal{M}^0(G; I, \Lambda; P) = \mathcal{M}(G; I, \Lambda; P) \cup \{0\}$, where $\mathcal{M}(G; I, \Lambda; P)$ is the semigroup $I \times G \times \Lambda$, with multiplication given by
\[(i,s,\lambda)(j,t,\mu) = (i,sp_{\lambda j}t,\mu).\]
It is well-known that $\mathcal{M}(G; I, \Lambda; P)$ is a \emph{completely simple} semigroup (i.e., one with a minimal idempotent in the natural partial order, but no proper ideals), and every completely simple semigroup is of this form (see, e.g.,~\cite[Theorem 3.3.1]{H}).

So in describing $\sim_s$ in the case where the sandwich matrix has only nonzero entries, there is no loss in generality in working with $\mathcal{M}(G; I, \Lambda; P)$ rather than $\mathcal{M}^0(G; I, \Lambda; P)$. The only difference is that the latter semigroup has one more $\sim_s$-equivalence class, consisting of just $0$. We shall rely on the well-known classification of congruences on $\mathcal{M}(G; I, \Lambda; P)$, a version of which we recall next.

\begin{theorem}[Theorem III.4.6 in \cite{PR}] \label{rees-cong-class}
Let $\mathcal{M}(G; I, \Lambda; P)$ be a completely simple Rees matrix semigroup, with appropriate $G$, $I$, $\Lambda$, and $P$, where $P$ is normalized $($i.e., contains a row and column where all the entries are $\, 1$$)$. A \emph{linked} or \emph{admissible triple} $(N, \mathcal{S}, \mathcal{T})$, consists of a normal subgroup $N$ of $G$ and equivalence relations $\mathcal{S}$, $\mathcal{T}$ on $I$, $\Lambda$, respectively, such that if $\, (i,j) \in \mathcal{S}$, then $p_{\lambda i}p_{\lambda j}^{-1} \in N$ for all $\lambda \in \Lambda$, and if $\, (\lambda, \mu) \in \mathcal{T}$, then $p_{\lambda i}p_{\mu i}^{-1} \in N$ for all $i\in I$. 

Given a linked triple $\, (N, \mathcal{S}, \mathcal{T})$, define a relation $\rho_{(N, \mathcal{S}, \mathcal{T})}$ on $\mathcal{M}(G; I, \Lambda; P)$ by 
\[(i,s,\lambda) \rho_{(N, \mathcal{S}, \mathcal{T})} (j,t,\mu)\] 
if $\, (i,j) \in \mathcal{S}$, $(\lambda, \mu) \in \mathcal{T}$, and $st^{-1} \in N$. Then $\rho_{(N, \mathcal{S}, \mathcal{T})}$ is a congruence on $\mathcal{M}(G; I, \Lambda; P)$. Conversely, given a congruence $\rho$ on $\mathcal{M}(G; I, \Lambda; P)$, we have $\rho = \rho_{(N, \mathcal{S}, \mathcal{T})}$ for a unique linked triple $\, (N, \mathcal{S}, \mathcal{T})$. 

Moreover, 
\[\mathcal{M}(G; I, \Lambda; P)/\rho_{(N, \mathcal{S}, \mathcal{T})}\cong \mathcal{M}(G/N; I/\mathcal{S}, \Lambda/\mathcal{T}; P/N),\]
where $P/N$ is the $\Lambda/\mathcal{T} \times I/\mathcal{S}$ matrix with $p_{\lambda i} N$ as the $\, (\mathcal{T}(\lambda), \mathcal{S}(i))$ entry $\, ($with $\mathcal{T}(\lambda)$ denoting the $\mathcal{T}$-equivalence class of $\lambda$, and likewise for $\mathcal{S}(i)$$)$.
\end{theorem}

It is well-known that every Rees matrix semigroup $\mathcal{M}(G; I, \Lambda; P)$ is isomorphic to one with a normalized sandwich matrix (see~\cite[Theorem 3.4.2]{H} or~\cite[Theorem III.2.6]{PR}). So there is no loss in generality in assuming that the sandwich matrix in normalized in the following description of $\sim_s$, which can also be viewed as a generalization of Corollary~\ref{group-sym}.

\begin{corollary} \label{rees-no-0-sym}
Let $\mathcal{M}(G; I, \Lambda; P)$ be a completely simple Rees matrix semigroup, with appropriate $G$, $I$, $\Lambda$, and $P$, where $P$ is normalized. Also let $H$ be the subgroup of $G$ generated by $\, [G,G]$ and the entries of $P$. Then for all $\, (i,s,\lambda), (j,t,\mu) \in \mathcal{M}(G; I, \Lambda; P)$, we have $\, (i,s,\lambda) \sim_s (j,t,\mu)$ if and only if $st^{-1} \in H$.
\end{corollary}

\begin{proof}
Since, by Theorem~\ref{least-comm}, $\sim_s$ is a congruence, we have $\sim_s \, = \rho_{(N, \mathcal{S}, \mathcal{T})}$ for a linked triple $(N, \mathcal{S}, \mathcal{T})$, by Theorem~\ref{rees-cong-class}.

Let $i,j \in I$ and $\lambda, \mu \in \Lambda$ be any elements. Then $(i,p_{\lambda i}^{-1},\lambda) \sim_p (j,p_{\lambda j}^{-1},\lambda)$ and $(i,p_{\lambda i}^{-1},\lambda) \sim_p (i,p_{\mu i}^{-1},\mu)$, by Theorem~\ref{rees-prim}, and so $(i,p_{\lambda i}^{-1},\lambda) \sim_s (j,p_{\lambda j}^{-1},\lambda)$ and $(i,p_{\lambda i}^{-1},\lambda) \sim_s (i,p_{\mu i}^{-1},\mu)$. Since $i,j \in I$ and $\lambda, \mu \in \Lambda$ were arbitrary, it follows that $\mathcal{S} = I \times I$ and $\mathcal{T} = \Lambda \times \Lambda$.

Let us next show that $H \subseteq N$. Taking any $i \in I$ and $\lambda \in \Lambda$, we can find $j \in I$ such that $p_{\lambda j} = 1$, since $P$ is normalized. Since $\mathcal{S} = I \times I$, and hence $(i,j) \in \mathcal{S}$, by Theorem~\ref{rees-cong-class}, we have $p_{\lambda i} = p_{\lambda i}p_{\lambda j}^{-1} \in N.$ Now let $i \in I$ and $\lambda \in \Lambda$ be such that $p_{\lambda i} = 1$, and let $p,r \in G$ be arbitrary. Then 
\[(i,prp^{-1}r^{-1},\lambda) = (i,p,\lambda)(i,r,\lambda)(i,p^{-1},\lambda)(i,r^{-1},\lambda),\]
and 
\[(i,1,\lambda) = (i,p,\lambda)(i,p^{-1},\lambda)(i,r,\lambda)(i,r^{-1},\lambda).\]
So $(i,prp^{-1}r^{-1},\lambda) \sim_s (i,1,\lambda)$, and hence $prp^{-1}r^{-1} = prp^{-1}r^{-1} \cdot 1^{-1} \in N$, by Theorem~\ref{rees-cong-class}. Since $N$ is a subgroup, it follows that $[G,G] \subseteq N$, and so $H \subseteq N$.

Given that $\mathcal{S} = I \times I$, $\mathcal{T} = \Lambda \times \Lambda$, and $H$ is a normal subgroup of $G$ (since it contains $[G,G]$) containing each $p_{\lambda i}$, we see that $(H, \mathcal{S}, \mathcal{T})$ is a linked triple. Hence, by Theorem~\ref{rees-cong-class}, $\mathcal{M}(G; I, \Lambda; P)/\rho_{(H, \mathcal{S}, \mathcal{T})} \cong G/H$. Since $[G,G] \subseteq H$, the group $G/H$ is abelian. Therefore, by Theorem~\ref{least-comm}, $\sim_s \, = \rho_{(N, \mathcal{S}, \mathcal{T})} \subseteq \rho_{(H, \mathcal{S}, \mathcal{T})}$. Since $H \subseteq N$, we also have $\rho_{(H, \mathcal{S}, \mathcal{T})} \subseteq \rho_{(N, \mathcal{S}, \mathcal{T})}$, and hence $\rho_{(N, \mathcal{S}, \mathcal{T})} = \rho_{(H, \mathcal{S}, \mathcal{T})}$. Thus, again by Theorem~\ref{rees-cong-class}, $H=N$, and so $(i,s,\lambda) \sim_s (j,t,\mu)$ if and only if $st^{-1} \in H$, for all $(i,s,\lambda), (j,t,\mu) \in \mathcal{M}(G; I, \Lambda; P)$.
\end{proof}

Note that the condition in Theorem~\ref{rees-prim}(2) characterizing when $(i,s,\lambda) \sim_p (j,t,\mu) \not\sim_p 0$, namely $rp_{\lambda i}s = tp_{\mu j}r$ for some $r \in G$, is indeed a special case of the condition in Corollary~\ref{rees-no-0-sym} characterizing when $(i,s,\lambda) \sim_s (j,t,\mu)$. Specifically, $rp_{\lambda i}s = tp_{\mu j}r$ is equivalent to $st^{-1} = (p_{\lambda i}^{-1}p_{\mu j})p_{\mu j}^{-1}(r^{-1}t)p_{\mu j}(rt^{-1})$, which is clearly an element of the subgroup $H$ from the corollary.

\section{Containment of Relations} \label{rel-comp-sect}

Using the observations above, we can completely describe the relationships between the various equivalence relations in Definitions~\ref{sim-def1}--\ref{sim-def2}, which we pause to do in this brief section.

First, we compare $\sim_s^1$ and $\sim_s$ to $\sim_o$, $\sim_w$, and $\sim_c$. As Proposition~\ref{free-semi-prop}, Proposition~\ref{conj-compar}, and~\cite[Theorem 2.2]{AKM} show, in any noncommutative free semigroup $F_{\Omega}$ we have
\[\sim_p^1 \, = \, \sim_p \, = \, \sim_c \, = \, \sim_w \, = \, \sim_o \, \subsetneq \, \sim_s^1 \, = \, \sim_s.\] 
On the other hand, in any commutative semigroup $\sim_s^1 \, = \, \sim_s$ (being the identity relation) is contained in each of $\sim_o$, $\sim_w$, $\sim_c$. To take a concrete example, in $S=\Z$, with multiplication given by $st = \min\{s,t\}$ ($s,t \in S$), we have $\sim_s^1 \, = \, \sim_s \, \subsetneq \, \sim_o \, = \, \sim_c$, since $\sim_o \, = \, \sim_c$ is the universal relation on $S$. Likewise, in any semigroup with trivial multiplication (i.e., $st = 0$ for all $s,t \in S$), $\sim_s^1 \, = \, \sim_s$ is the identity relation, but $\sim_w$ is the universal relation. Using these observations, it is easy to construct semigroups where $\sim_s$ and $\sim_s^1$ are incomparable with $\sim_o$, $\sim_w$, and $\sim_c$.

\begin{example} \label{incomp-eg}
Let $\approx \ \in \{\sim_o, \sim_w, \sim_c\}$, and let $S$ be a semigroup for which $\sim_s^1 \, = \, \sim_s$ is the identity relation, $\approx$ is the universal relation, and $\sim_s \, \neq \, \approx$. Also, let $\Omega$ be a set of cardinality at least $3$, and let $F_{\Omega}$ be the free semigroup on $\Omega$. Now take $T = F_{\Omega} \times S$, let $\alpha,\beta,\gamma \in \Omega$ be distinct, and let $s,t \in S$ be distinct. Then in $T$, $(\alpha,s) \approx (\alpha,t)$, but $(\alpha,s) \not\sim_s (\alpha,t)$. On the other hand, $(\gamma \alpha \beta,s^3) \sim_s^1 (\gamma \beta \alpha,s^3)$, but $(\gamma \alpha \beta,s^3) \not\approx (\gamma \beta \alpha,s^3)$, since $\approx \, = \, \sim_p^1$ in $F_{\Omega}$.
\end{example}

We are now ready to explain how all the aforementioned relations interact. As mentioned in Section~\ref{def-sect}, in any semigroup, we have $\sim_p^1 \, \subseteq \, \sim_s^1 \, \subseteq \, \sim_s$, $\sim_p \, \subseteq \, \sim_s$, $\sim_n \, \subseteq \, \sim_p^1 \, \subseteq \, \sim_p \, \subseteq \, \sim_w \, \subseteq \, \sim_o$, and $\sim_n \, \subseteq \, \sim_c \, \subseteq \, \sim_o$. Generally speaking, $\sim_c$ is not comparable to $\sim_p^1$ or $\sim_p$~\cite[Section 1]{AKKM1}, $\sim_s^1$ and $\sim_s$ are not comparable to $\sim_o$ and $\sim_w$ and $\sim_c$ (Example~\ref{incomp-eg}), $\sim_s^1 \, \not\subseteq \, \sim_p$ (Proposition~\ref{free-semi-prop}), and $\sim_p \, \not\subseteq \, \sim_s^1$ (Example~\ref{rees-eg}). Finally, there are semigroups for which $\sim_c \, = \, \sim_o \, \not\subseteq \, \sim_w$, by~\cite[Theorem 3.6]{ABKKMM}, and $\sim_w \, \not\subseteq \, \sim_c$ in any nonzero semigroup with trivial multiplication. So we can illustrate the containments among the relations in question as follows.

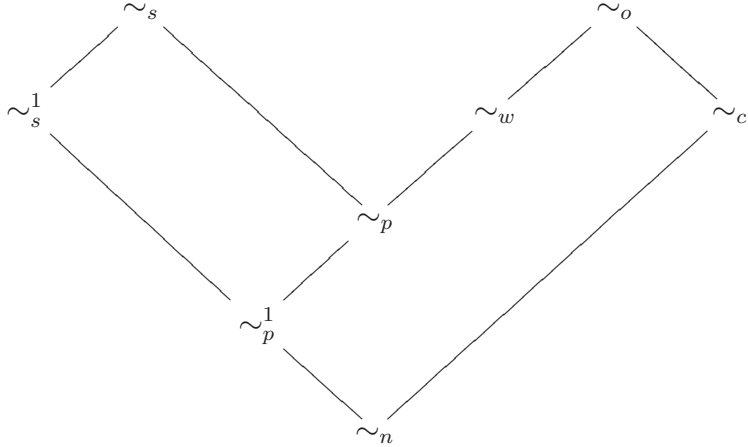
\[\xymatrix{
& \sim_s \ar@{-}[ddrr] \ar@{-}[dl] & & & & \sim_o \ar@{-}[dl] \ar@{-}[dr] & \\
\sim_s^1 \ar@{-}[ddrr] & & & & \sim_w \ar@{-}[dl] & & \sim_c \ar@{-}[dddlll] \\
& & & \sim_p \ar@{-}[dl] & & &\\
& & \sim_p^1 \ar@{-}[dr] & & & &\\
& & & \sim_n & & &\\
}\]

All the above containments are generally strict, which justifies having eight separate relations in the diagram. Specifically, there are semigroups where $\sim_c \, \neq \, \sim_o$ (Proposition~\ref{conj-compar}), $\sim_p \, \neq \, \sim_s$, $\sim_p^1 \, \neq \, \sim_s^1$ (Example~\ref{group-eg} or Proposition~\ref{free-semi-prop}), $\sim_p^1 \, \neq \, \sim_p$ (Theorem~\ref{rees-prim}), $\sim_s^1 \, \neq \, \sim_s$ (Example~\ref{rees-eg}), and $\sim_w \, \neq \, \sim_o$ \cite[Theorem 3.6]{ABKKMM}. As mentioned in Section~\ref{free-sect}, $\sim_n$ is the identity relation on a free semigroup, and so $\sim_n \, \neq \, \sim_p^1 \, =  \, \sim_c$ in any noncommutative free semigroup.

Finally, explorations of the relationship between $\sim_p$ and $\sim_w$ have a very interesting history. Let $M$ denote the semigroup of all infinite matrices, with rows and columns indexed by $\Z^+$, entries from $\N$, and finite support (i.e., only finitely many nonzero entries), under the usual matrix multiplication. As alluded to in Sections~\ref{intro} and~\ref{def-sect}, when applied to $M$, in the context of symbolic dynamics, $\sim_p$ is known as \emph{strong shift equivalence}, and $\sim_w$ as \emph{shift equivalence}. The question of whether $\sim_p$ and $\sim_w$ coincide on $M$ was open for more than twenty years, eventually being resolved in the negative by Kim and Roush~\cite{KR}. Of course, one can find simpler examples of semigroups where $\sim_p \, \neq \, \sim_w$, if desired.

\section{Graph Inverse Semigroups} \label{graph-inv-semi-sect}

A (\emph{directed}) \emph{graph} $E=(E^0,E^1,\ra,\so)$ consists of two sets $E^0,E^1$ (containing \emph{vertices} and \emph{edges}, respectively), together with functions $\so,\ra:E^1 \to E^0$, called \emph{source} and \emph{range}, respectively. A \emph{path} $x$ in $E$ is a finite sequence of (not necessarily distinct) edges $x=e_1\cdots e_n$ such that $\ra(e_i)=\so(e_{i+1})$ for $i=1,\dots,n-1$. In this case, $\so(x):=\so(e_1)$ is the \emph{source} of $x$, $\ra(x):=\ra(e_n)$ is the \emph{range} of $x$, and $|x|:=n$ is the \emph{length} of $x$. A path $x$ is \emph{closed} if $\so(x)=\ra(x)$, while a closed path consisting of just one edge is called a \emph{loop}. We view the elements of $E^0$ as paths of length $0$ (extending $\so$ and $\ra$ to $E^0$ via $\so(v)=v=\ra(v)$ for all $v\in E^0$), and denote by $\pth(E)$ the set of all paths in $E$, and by $\clpth(E)$ the set of all closed paths in $E$.

Given a graph $E=(E^0,E^1,\ra,\so)$, the \emph{graph inverse semigroup $G(E)$ of $E$} is the semigroup with zero generated by $E^0$ and $E^1$, together with $E^{-1}:= \{e^{-1} \mid e\in E^1\}$, satisfying the following relations for all $v,w\in E^0$ and $e,f\in E^1$:\\
(V)  $vw = \delta_{v,w}v$,\\ 
(E1) $\so(e)e=e\ra(e)=e$,\\
(E2) $\ra(e)e^{-1}=e^{-1}\so(e)=e^{-1}$,\\
(CK1) $e^{-1}f=\delta _{e,f}\ra(e)$.\\
(Here $\delta$ is the Kronecker delta.) We define $v^{-1}=v$ for each $v \in E^0$, and for any path $x=e_1\cdots e_n$ ($e_1,\dots, e_n \in E^1$) we let $x^{-1} = e_n^{-1} \cdots e_1^{-1}$. With this notation, every nonzero element of $G(E)$ can be written uniquely as $xy^{-1}$ for some $x, y \in \pth(E)$, where $\ra(x) = \ra(y)$. It is also easy to verify that $G(E)$ is indeed an inverse semigroup, with $(xy^{-1})^{-1} = yx^{-1}$ for all $x, y \in \pth (E)$.

If $E$ is a graph with only one vertex and $n$ edges (necessarily loops), for some $n \in \Z^+$, then $G(E)$ is known as a \emph{polycyclic monoid} (or the \emph{bicyclic monoid}, if $n=1$).

For polycyclic monoids, the relation $\sim_p^1$ is characterized in~\cite[Theorem 3.6]{AKKM2}, $\sim_c$ in~\cite[Theorem 3.9]{AKKM2}, and $\sim_n$ in~\cite[Theorem 5.2]{ABKKMM}. Also the relation $\sim_p$ is characterized for all graph inverse semigroups in~\cite{MV}. We record that result here, along with the necessary terminology, give a more convenient restatement, and then use it  to characterize $\sim_s$.

\begin{definition} \label{approx-def}
Let $E$ be a graph, and $x, y \in \clpth(E)$. We write $x \approx y$  if there exist $z_1,z_2 \in \pth(E)$ such that $x=z_1z_2$ and $z_2z_1=y$.
\end{definition}

It is shown in \cite[Lemma 12]{M3} that $\approx$ is an equivalence relation.

\begin{proposition}[Proposition 20 in~\cite{MV}] \label{eq-description}
Let $E$ be a graph, and for each $x \in \clpth (E)$ set 
\[EQ(x) := \{yzy^{-1} \mid y \in \pth (E), z \in \clpth (E), \ra(y)=\so(z), z \approx x\} \text{ and}\]
\[EQ(x^{-1}) := \{yz^{-1}y^{-1} \mid y \in \pth (E), z \in \clpth (E), \ra(y)=\ra(z), z \approx x\}.\]
Then every nonzero $\sim_p$-equivalence class of $G(E)$ is of the form $EQ(x)$ or $EQ(x^{-1})$ for some $x \in \clpth (E)$.

In particular, for all $x_1, x_2 \in \clpth (E)$ we have $EQ(x_1) \cap EQ(x_2) \neq \emptyset$ if and only if $x_1 \approx x_2$ if and only if $EQ(x_1^{-1}) \cap EQ(x_2^{-1}) \neq \emptyset$, and $EQ(x_1) \cap EQ(x_2^{-1}) \neq \emptyset$ if and only if $x_1=x_2 \in E^0$.
\end{proposition}

\begin{corollary} \label{graph-p-description}
Let $E$ be a graph, and $s, t \in G(E)$. Then $s \sim_p t$ if and only if exactly one of the following holds.
\begin{enumerate}
\item[$(1)$] There exist $x_1,x_2 \in \clpth (E)$ and $y,z \in \pth (E)$ such that $x_1 \approx x_2$, $\ra(y)=\so(x_1)$, $\ra(z)=\so(x_2)$, $s = yx_1y^{-1}$, and $t = zx_2z^{-1}$.
\item[$(2)$] There exist $x_1,x_2 \in \clpth (E)\setminus E^0$ and $y,z \in \pth (E)$ such that $x_1 \approx x_2$, $\ra(y)=\ra(x_1)$, $\ra(z)=\ra(x_2)$, $s = yx_1^{-1}y^{-1}$, and $t = zx_2^{-1}z^{-1}$.
\item[$(3)$] Neither $s$ nor $t$ is of the form $yxy^{-1}$ or $yx^{-1}y^{-1}$, for any $x \in \clpth (E)$ and $y \in \pth (E)$. $($This case occurs if and only if $s \sim_p 0 \sim_p t$.$)$
\end{enumerate}
\end{corollary}

\begin{proof}
It follows immediately from Proposition~\ref{eq-description} that $s \sim_p 0 \sim_p t$ if and only if $s$ and $t$ are not of the form $yxy^{-1}$ or $yx^{-1}y^{-1}$, for any $x \in \clpth (E)$ and $y \in \pth (E)$, and if $s \not\sim_p 0 \not\sim_p t$, then $s \sim_p t$ if and only if $s$ and $t$ satisfy either (1) or (2). In (2) we insist on $x_1$ and $x_2$ not being vertices, to ensure that $s$ and $t$ cannot satisfy (1) and (2) simultaneously.
\end{proof}

For the remainder of the section we employ the convention that for any loop $e$ in a graph $E$ and any $n \in \Z$, $e^n$ denotes the product of $n$ copies of $e$ if $n >0$, the product of $|n|$ copies of $e^{-1}$ if $n<0$, and $e^n = \so(e)$ if $n=0$. 

To describe $\sim_s$ in graph inverse semigroups, we require a technical lemma. 

\begin{lemma} \label{graph-lem}
Let $E$ be a graph, and $v\in E^0$. Then the following hold.
\begin{enumerate}
\item[$(1)$] If $\, \ra^{-1}(v) = \{v\}$, then the $\sim_s$-equivalence class of $v$ is $\, \{v\}$.
\item[$(2)$] If $\, \ra^{-1}(v) = \{v,e\}$ for some loop $e\in E^1$, then the $\sim_s$-equivalence class of $v$ is $\, \{e^ne^{-n} \mid n \in \N\}$.
\end{enumerate}
\end{lemma}

\begin{proof}
(1) If $\ra^{-1}(v) = \{v\}$, then, by Corollary~\ref{graph-p-description} (or Proposition~\ref{eq-description}), the $\sim_p$-equivalence class of $v$ is $\{v\}$. This implies that the only way to express $v$ as a product of elements of $G(E)$ is $v = v \cdots v$, and so the $\sim_s$-equivalence class of $v$ is $\{v\}$ as well. 

(2) Suppose that $\ra^{-1}(v) = \{v,e\}$ for some loop $e\in E^1$. Then, by Corollary~\ref{graph-p-description} (or Proposition~\ref{eq-description}), the $\sim_p$-equivalence class of $v$ is $\{e^ne^{-n} \mid n \in \N\}$, and hence this set is contained in the $\sim_s$-equivalence class of $v$. This also implies that if $v = g_1\cdots g_n$ for some $g_i \in E^0 \cup E^1 \cup E^{-1}$, then each $g_i \in \{v, e, e^{-1}\}$, and the number of copies of $e$ among the $g_i$ is equal to the number of copies of $e^{-1}$. It follows that if $v \sim_s^1 s$ for some $s \in G(E)$, then $s = e^ne^{-n}$ for some $n \in \N$. Iterating this argument (on $e^ne^{-n}$) shows that if $v \sim_s s$ for some $s \in G(E)$, then $s = e^ne^{-n}$ for some $n \in \N$.
\end{proof}

\begin{theorem} \label{graph-semi-sim}
Let $E$ be a graph, and $s, t \in G(E)$. Then $s \sim_s t$ if and only if exactly one of the following holds.
\begin{enumerate}
\item[$(1)$] There exists a vertex $v \in E^0$ such that $\, \ra^{-1}(v) = \{v\}$ and $s = v = t$. 
\item[$(2)$] There exist a loop $e \in E^1$ and $n_1,m_1,n_2,m_2 \in \N$ such that $s = e^{n_1}e^{-m_1}$, $t = e^{n_2}e^{-m_2}$, $n_1-m_1=n_2-m_2$, and $\, \ra^{-1}(\so(e)) = \{\so(e), e\}$.
\item[$(3)$] Neither $s$ nor $t$ is of the forms described in $\, (1)$ and $\, (2)$. $($This case occurs if and only if $s \sim_s 0 \sim_s t$.$)$
\end{enumerate}
\end{theorem}

\begin{proof}
If $s$ and $t$ satisfy (1), then, certainly, $s \sim_s t$. Moreover, by Lemma~\ref{graph-lem}(1), in this case the $\sim_s$-equivalence class of $s=t$ does not contain $0$.

Now suppose that $s$ and $t$ satisfy (2). Then
\[s \sim_s  e^{-m_1}e^{n_1} = e^{n_1-m_1} =  e^{n_2-m_2} =  e^{-m_2}e^{n_2} \sim_s t.\] 
Moreover, by Lemma~\ref{graph-lem}(2), in this situation the $\sim_s$-equivalence class of $\so(s) = \so(t)$ does not contain $0$. We claim that the $\sim_s$-equivalence class of $s$ and $t$ does not either. For suppose that $e^{n_1-m_1} \sim_s 0$. Since $\sim_s$ is a congruence (by Theorem~\ref{least-comm}), this would give $e^{m_1-n_1}e^{n_1-m_1} \sim_s e^{m_1-n_1}\cdot 0$ and $e^{n_1-m_1}e^{m_1-n_1} \sim_s 0 \cdot e^{m_1-n_1}$. But either $\so(e) = e^{m_1-n_1}e^{n_1-m_1}$ or $\so(e) = e^{n_1-m_1}e^{m_1-n_1}$, and so we would have $\so(e) \sim_s 0$, producing a contradiction. Thus the $\sim_s$-equivalence class of $s$ and $t$ does not contain $0$.

Next suppose that $s$ does not satisfy (1) or (2). We may further suppose that $s \neq 0$, since otherwise $s \sim_s 0$. Write $s = xy^{-1}$ for some $x,y \in \pth(E)$, and let $v = \ra(x)$. Then either there are distinct loops $e_1,e_2 \in E^1$ such that $\so(e_1)=v=\so(e_2)$, or there exists $g \in E^1$ such that $\ra(g) = v$ and $\so(g) \neq v$. In the first case, 
\[e_1 = e_1e_2^{-1}e_2 \sim_s e_2^{-1}e_1e_2 = 0.\]
Since $\sim_s$ is a congruence, we have $v = e_1^{-1}e_1 \sim_s e_1^{-1}\cdot 0 = 0$, which gives $s = xvy^{-1} \sim_s 0$. In the second case, $g = gv \sim_s vg = 0$, which again gives $v = g^{-1}g \sim_s 0$ and $s \sim_s 0$. It follows that if $s$ and $t$ satisfy (3), then $s \sim_s 0 \sim_s t$. 

Conversely, suppose that $s \sim_s t$. If $s \sim_s 0 \sim_s t$, then by the first two paragraphs of this proof, $s$ and $t$ satisfy (3). (In particular, this establishes the parenthetical claim in (3).) So let us assume that $s \not\sim_s 0 \not\sim_s t$, and let $v = \so(s)$. Then $s = vs$ implies that $v \not\sim_s 0$, since $\sim_s$ is a congruence. Therefore $s = vs \sim_s vt$, which implies that $v = \so(t)$. Now suppose that there exist $w \in E^0 \setminus \{v\}$ and $e \in E^1$, such that $\so(e) = w$ and $\ra(e) = v$. Then $e = we \sim_s ew = 0$, and so $v = e^{-1}e \sim_s 0$, producing a contradiction. Thus $\so(e)=v$ for all $e \in E^1$ with $\ra(e) = v$. A similar argument shows that if there exists $e \in E^1$ such that $\so(e) = v$ and $\ra(e) \neq v$, then $e \sim_s 0$. In particular, writing $s = xy^{-1}$ for some $x,y \in \pth(E)$, it follows that $\ra(y) = \ra(x) = v = \so(y)$, and likewise for $t$.

Next suppose that $e_1,e_2 \in E^0$ are distinct loops satisfying $\so(e_1) = v = \so(e_2)$. Then, as before, 
\[e_1 = e_1e_2^{-1}e_2 \sim_s e_2^{-1}e_1e_2 = 0,\]
which gives $v = e_1^{-1}e_1 \sim_s 0$ and $s \sim_s 0$. Thus either $\ra^{-1}(v) = \{v\}$, or $\ra^{-1}(v) = \{v, e\}$ for some loop $e \in E^1$. In the first case, $s = v = t$; i.e., $s$ and $t$ satisfy (1). In the second case, necessarily, $s = e^{n_1}e^{-m_1}$ and $t = e^{n_2}e^{-m_2}$ for some $n_1,m_1,n_2,m_2 \in \N$. Clearly, $s \sim_s e^{n_1-m_1}$ and $t \sim_s e^{n_2-m_2}$. Using the fact that $\sim_s$ is a congruence once more, we see that 
\[v \sim_s e^{n_1-m_1}e^{m_1-n_1} \sim_s e^{n_2-m_2}e^{m_1-n_1}\sim_s e^{m_1-n_1 + n_2-m_2}.\]
Lemma~\ref{graph-lem}(2) then implies that $n_1-m_1=n_2-m_2$, and so $s$ and $t$ satisfy (2). Thus, if $s \sim_s t$, then $s$ and $t$ must satisfy one of (1)--(3). Moreover, those three conditions are clearly mutually exclusive.
\end{proof}

We could have used the fact that $\sim_s$ is the least commutative congruence on any semigroup (Theorem~\ref{least-comm}) for the second half of the proof above, instead of the more direct approach taken. Specifically, one can check, either directly or via the results in~\cite{W} (which describe all the congruences on a graph inverse semigroup), that identifying the elements of $G(E)$ according to (1)--(3) above produces a commutative congruence on $G(E)$. From this it follows that $s \sim_s t$ implies that $s$ and $t$ both satisfy the same condition among (1)--(3), for all $s,t \in G(E)$.

\section{Classical Transformation Semigroups}

Given a set $\Omega$, we denote by $\T(\Omega)$ the monoid of all functions from $\Omega$ to $\Omega$, by $\PT(\Omega)$ the monoid of all partial functions from $\Omega$ to $\Omega$, and by $\I(\Omega)$ the symmetric inverse monoid on $\Omega$. It turns out that $\sim_s$ can be described in exactly the same way on these three semigroups, and so we shall do that simultaneously.

Before continuing, we recall some terminology pertaining to partial transformations. Given a set $\Omega$, the elements of $\PT(\Omega)$ are functions $s: \Gamma\rightarrow \Delta$, where $\Gamma,\Delta \subseteq \Omega$, and the elements of $\I(\Omega)$ are the bijective functions in $\PT(\Omega)$. Here we let $\dom(s):=\Gamma$ be the \emph{domain} of $s$, and $\im(s): = \Delta$ be the \emph{image} of $s$. For all $s,t \in \PT(\Omega)$, $st \in \PT(\Omega)$ is taken to be the composite of $s$ and $t$ as functions, restricted to the domain $t^{-1}(\dom(s)\cap \im(t))$. 

In the semigroups $\T(\Omega)$, $\PT(\Omega)$, and $\I(\Omega)$ the various relations in Definitions~\ref{sim-def1} and~\ref{sim-def2} have been studied extensively. Let us review the relevant literature, for the convenience of the reader. In $\T(\Omega)$, the relation $\sim_c \, = \, \sim_o$  is classified in~\cite[Theorem 6.1]{AKM}, and a description of $\sim_n$ is given in~\cite[Theorem 4.11]{K} and~\cite[Theorem 2.33]{ABKKMM}. For finite $\Omega$, the relation $\sim_p$ is classified in~\cite[Theorem 1]{KM}. In $\PT(\Omega)$, the relation $\sim_c \, = \, \sim_o$  is classified in~\cite[Theorem 5.3]{AKM}, and $\sim_n$ is described in~\cite[Theorem 4.8]{K}. For finite $\Omega$, the relation $\sim_p$ is classified in~\cite[Theorem 1]{KM}. In $\I(\Omega)$, with $\Omega$ countable, $\sim_p$ is classified in~\cite[Theorem 2]{KM}, and $\sim_c$ is classified in~\cite[Theorem 2.14]{AKKM1}. By~\cite[Corollary 5.2]{K} and~\cite[Proposition 2]{KM}, $\sim_n \, = \, \sim_p$ in $\I(\Omega)$, for all $\Omega$. 

In each case where there is a complete classification of equivalence classes in the aforementioned semigroups, in terms of the actions of the elements, it tends to be rather difficult to state and prove. In contrast to this, we can obtain complete descriptions of the $\sim_s$-equivalence classes in these semigroups rather quickly, but at the cost of the result being more trivial.

For $\T(\Omega)$, $\PT(\Omega)$, and $\I(\Omega)$  there are well-known complete classifications of congruences--see~\cite{Ma} (alternatively, \cite[\S 10.8]{CP}), \cite{S}, and~\cite{L}, respectively. Our strategy in describing $\sim_s$ in these semigroups is, fundamentally, to rely on those classifications. For infinite $\Omega$ the congruence classifications are somewhat complicated, and so we shall handle the infinite cases more directly, with the help of the next lemma. The first statement in this lemma is a variation on~\cite[Theorem~3.3]{HRH}, which says that for infinite $\Omega$, the semigroup $\T(\Omega)$ is generated by the symmetric group $\B(\Omega)$ together with an injection and a surjection.

\begin{lemma} \label{transf-gen}
The following hold for any infinite set $\, \Omega$.
\begin{enumerate}
\item[$(1)$] There exist $s,t \in \T(\Omega)$, satisfying $st=1$, such that $\T(\Omega) = s\B(\Omega)t$.
\item[$(2)$] If $s \in \I(\Omega)$ is such that $\, \dom(s) = \Omega$ and $\, |\Omega \setminus \im(s)| = |\Omega|$, then $\I(\Omega) = s^{-1}\B(\Omega)s$.
\end{enumerate}
\end{lemma}

\begin{proof}
(1) Since $\Omega$ is infinite, we can write $\Omega = \bigcup_{\alpha \in \Omega} \Sigma_\alpha$, where the union is disjoint, and $|\Sigma_\alpha| = |\Omega|$ for
each $\alpha \in \Omega$. Let $s,t \in \T(\Omega)$ be such that $s(\Sigma_\alpha) = \alpha$ and $t(\alpha) \in \Sigma_\alpha$ for each $\alpha \in \Omega$. Then $st=1$.

Now let $p \in \T(\Omega)$ be any element, and for each $\alpha \in \Omega$
let $\Delta_\alpha \subseteq \Omega$ denote the preimage $p^{-1}(\alpha)$ of $\alpha$ under $p$. We can find an injective $q \in \T(\Omega)$ that embeds $\Delta_\alpha$ in $\Sigma_\alpha$, for each $\alpha \in \Omega$, with the property that $|\Sigma_\alpha \setminus p(\Delta_\alpha)| = |\Omega|$ for some $\alpha \in \Omega$. Then $p = sq$. Since $t$ and $q$ are both injective and 
\[|\Omega \setminus t(\Omega)| = |\Omega| = |\Omega \setminus q(\Omega)|,\] 
there exists $r \in \B(\Omega)$ such that $rt(\alpha) = q(\alpha)$ for all $\alpha \in \Omega$. Hence $p = srt \in s\B(\Omega)t$, and so $\T(\Omega) = s\B(\Omega)t$.

(2) Let $s \in \I(\Omega)$ be as in the statement, and let $p \in \I(\Omega)$ be any element.  Then
\[|s(\dom(p))| = |\dom(p)| = |\im(p)| = |s(\im(p))|,\]
and
\[|\im(s) \setminus s(\dom(p))| \leq |\Omega| = |\Omega \setminus \im(s)|.\]
So there exists $r \in \B(\Omega)$ that takes $s(\dom(p))$ to $s(\im(p))$, and takes $\im(s) \setminus s(\dom(p))$ into $\Omega \setminus \im(s)$. Since $\dom(s^{-1}) = \im(s)$, we have $\dom(s^{-1}rs) = \dom(p)$ and $\im(s^{-1}rs) = \im(p)$. Clearly, we can choose $r$ so that $s^{-1}rs = p$, and so $p \in s^{-1}\B(\Omega)s$.
\end{proof}

\begin{proposition}\label{full-transf-sim}
Let $\, \Omega$ be a set. If $\, \Omega$ is infinite, then $\sim_s$ is the universal relation on $\T(\Omega)$, $\PT(\Omega)$, and $\I(\Omega)$. If $\, \Omega$ is finite, then in each of these semigroups there are three $\sim_s$-congruence classes--consisting of even permutations of $\, \Omega$, odd permutations of $\, \Omega$, and $($partial$)$ transformations with image of size $<|\Omega|$.
\end{proposition}

\begin{proof}
First suppose that $\Omega$ is infinite. Let $S$ denote $\T(\Omega)$ or $\I(\Omega)$, and let $p \in S$. Then, by Lemma~\ref{transf-gen}, $p = sqt$, for some $s,t \in S$ such that $st=1$, and some $q \in \B(\Omega)$. Now, by~\cite[Theorem 6]{O}, there exist $r_1,r_2 \in \B(\Omega)$ such that $q = r_1r_2r_1^{-1}r_2^{-1}$. Thus 
\[p = s(r_1r_2r_1^{-1}r_2^{-1})t \sim_s (st)(r_1r_1^{-1})(r_2r_2^{-1}) = 1.\]
So we conclude that the $\sim_s$-equivalence class of $1$ is all of $S$.

Next let $p\in \PT(\Omega)$, let $r_1 \in \T(\Omega) \subseteq \PT(\Omega)$ be such that $r_1$ agrees with $p$ on $\dom(p)$ and acts arbitrarily (e.g., as the identity) on $\Omega \setminus \dom(p)$, and let $r_2 \in \I(\Omega) \subseteq \PT(\Omega)$ be such that $\dom(r_2) = \dom(p)$ and $r_2$ acts as the identity on $\dom(p)$. Then $p = r_1r_2$. By Lemma~\ref{transf-gen}, there exist $s_1,t_1 \in \T(\Omega)$, $s_2,t_2 \in \I(\Omega)$, and $q_1,q_2 \in \B(\Omega)$ such that $p = (s_1q_1t_1)(s_2q_2t_2)$ and $s_1t_1=1=s_2t_2$. Hence $p \sim_s q_1q_2 \in \B(\Omega)$, and so, as in the previous paragraph, $p \sim_s 1$. Thus $\sim_s$ is the universal relation on $\PT(\Omega)$ as well.

Now suppose that $\Omega$ is finite, and let $S$ denote any of $\T(\Omega)$, $\PT(\Omega)$, or $\I(\Omega)$. In this case the classification of the congruences on $S$ is simpler, and can be stated in the same way for any of the three semigroups in question--see~\cite[Theorem 6.3.10]{GM}, or~\cite[Theorem 2.2]{ABG} for an even more succinct account, which we shall not attempt to reproduce here. By Corollary~\ref{group-sym}, if $s,t \in \B(\Omega)$ are such that $st^{-1} \in [\B(\Omega), \B(\Omega)]$, then $s \sim_s t$, both in $\B(\Omega)$ and in $S$. By~\cite[Theorem 1]{O}, $[\B(\Omega), \B(\Omega)]$ is the alternating subgroup of $\B(\Omega)$, and so in $S$, $\sim_s$  must relate all odd permutations of $\Omega$ and relate all even permutations of $\Omega$. By~\cite[Theorem 6.3.10]{GM}, the only non-universal congruence on $S$ that has this property is the congruence that also relates all elements with image of size $< |\Omega|$, i.e., all elements of $S \setminus \B(\Omega)$. It is easy to see that taking the quotient of $S$ by this congruence gives a commutative semigroup (with $3$ elements), and hence $\sim_s$ must be the congruence in question, by Theorem~\ref{least-comm}.
\end{proof}

\section{Injective Function Semigroups}

Given a set $\Omega$, we denote by $\J(\Omega)$ the monoid of all injective functions from $\Omega$ to $\Omega$. The relation $\sim_o \, = \, \sim_c$ in this semigroup is characterized in~\cite[Theorem 7.6]{AKM}, and $\sim_n$ in~\cite[Theorem 5.3]{K}. For $\sim_p^1$ a characterization is available only for countable $\Omega$--see the remarks following~\cite[Lemma 3.3]{AKKM1}. So we shall classify $\sim_p^1$ and $\sim_p$ in $\J(\Omega)$ for all $\Omega$, before doing the same for $\sim_s$. We require additional terminology, some of which we can state in the more general context of the full transformation semigroup $\T(\Omega)$ without much loss of efficiency.

We say that a (directed) graph $E=(E^0,E^1,\ra,\so)$ is \emph{simple} if for all $v, w \in E^0$ there is at most one edge $e \in E^1$ such that $\so(e)=v$ and $\ra(e)=w$ (see Section~\ref{graph-inv-semi-sect} for the notation). In this situation one can view $E^1$ as simply a binary relation on $E^0$, where $(u,v) \in E^1$ if there is an edge with source $u$ and range $v$, for all $u,v \in E^0$. From now on we shall use the notation $E=(E^0,E^1)$ for simple graphs, and interpret $E^1$ in this manner. Note that here we permit loops (i.e., edges of the form $(v,v)$), but this is not a standard convention for simple graphs. 

A \emph{strongly connected component} of a simple graph $E=(E^0,E^1)$ is a (directed) subgraph $F$ maximal with respect to the property that for all distinct $u,v \in F^0$ there is a path from $u$ to $v$. A \emph{weakly connected component} of $E$ is a subgraph which results in a strongly connected component in the graph $(E^0, \overline{E^1})$, where $\overline{E^1}$ is the symmetric closure of $E^1$.

Let $E_a=(E_a^0,E^1_a)$ and $E_b=(E_b^0,E^1_b)$ be two simple graphs. A function $f : E_a^0 \to E_b^0$ is a \emph{graph homomorphism from} $E_a$ \emph{to} $E_b$ if for all $u,v \in E_a^0$, $(u,v) \in E^1_a$ implies that $(f(u),f(v)) \in E^1_b$. Such a function is a \emph{graph isomorphism} if it is bijective, and for all $u,v \in E_a^0$, $(u,v) \in E^1_a$ if and only if $(f(u),f(v)) \in E^1_b$. In this situation we write $E_a \cong E_b$.

When describing conjugacy classes in transformation semigroups on a set $\Omega$, it is often convenient to represent each transformation as a directed graph. (See, e.g.,~\cite[\S 3]{KM}.) Specifically, given $s \in \T(\Omega)$ let $E(s) = (E^0,E^1)$ be the simple graph where $E^0 = \Omega$, and $(\alpha, \beta) \in E^1$ whenever $s(\alpha) = \beta$. 

\begin{definition} \label{conn-comp-def}
Let $\, \Omega$ be a set, $\, \Sigma \subseteq \Omega$ nonempty, and $s \in \T(\Omega)$. We say that $\, \Sigma$ is a \emph{connected component} of $s$ if the following two conditions are satisfied:
\begin{enumerate}
\item[$($\rm{i}$)$] $s(\alpha) \in \Sigma$ if and only if $\alpha \in \Sigma$, for all $\alpha \in \Omega;$ 
\item[$($\rm{ii}$)$] $\Sigma$ has no proper nonempty subset satisfying $\, ($\rm{i}$)$.
\end{enumerate}
\end{definition}

It is easy to see that a connected component of $s \in \T(\Omega)$ corresponds to a weakly connected component in the associated graph $E(s)$. The next lemma gives a stronger version of this observation, as well as a description of the connected components of the elements of $\J(\Omega)$.

\begin{lemma} \label{conn-comp-lemm}
The following hold for any set $\, \Omega$.
\begin{enumerate}
\item[$(1)$] Let $s \in \T(\Omega)$, and $\alpha, \beta \in \Omega$. Then $\alpha$ and $\beta$ belong to the same connected component of $s$ if and only if $s^n(\alpha) = s^m(\beta)$ for some $n,m \in \N$.
\item[$(2)$] Let $s \in \J(\Omega)$, and $\alpha \in \Omega$. Then 
\begin{equation}\label{eq-0} \tag{\dag}
\{s^n(\alpha) \mid n \in \N\}\cup \{\beta \in \Omega \mid \exists n \in \Z^+ \, (s^n(\beta) = \alpha)\}
\end{equation} 
is a connected component of $s$, and every connected component of $s$ is of this form.
\end{enumerate}
\end{lemma}

\begin{proof}
(1) Clearly, we can find a connected component $\Sigma \subseteq \Omega$ of $s$ such that $\alpha \in \Sigma$. Now suppose that $s^n(\alpha) = s^m(\beta)$ for some $n,m \in \N$. Then $s^n(\alpha) \in \Sigma$, and so $s^m(\beta) \in \Sigma$. Hence $s^{m-1}(\beta) \in \Sigma$ (in case $m>1$), and therefore, by induction, $\beta \in \Sigma$.

Conversely, suppose that $\alpha, \beta \in \Sigma$, for some connected component $\Sigma$ of $s$. Define recursively $\Gamma_0(\alpha) = \{s^n(\alpha) \mid n\in \N\}$, and  $\Gamma_{-m}(\alpha) = s^{-1}(\Gamma_{-(m-1)})$ for all $m > 0$. Also let $\Gamma(\alpha) = \bigcup_{m=0}^{-\infty} \Gamma_m$. Then, clearly, $\Gamma(\alpha) \subseteq \Sigma$, and $\Gamma(\alpha)$ is a connected component of $s$. Therefore $\Gamma(\alpha) = \Sigma$, by Definition~\ref{conn-comp-def}. It follows that $s^m(\beta) \in \Gamma_0(\alpha)$ for some $m \in \N$, and so $s^n(\alpha) = s^m(\beta)$ for some $n \in \N$.

(2) By (1), the set in (\ref{eq-0}) is contained in a connected component of $s$. Since $s$ is injective, the set in (\ref{eq-0}) also contains all $\beta \in \Omega$ such that $s^n(\alpha) = s^m(\beta)$ for some $n,m \in \N$, and hence must be a connected component of $s$, again by (1). Since each $\alpha \in \Omega$ belongs to a connected component of $s$, and, clearly, any such connected component must contain the set in (\ref{eq-0}), it follows that every connected component of $s$ is of this form.
\end{proof}

In particular, every connected component of $s \in \J(\Omega)$ must be countable, and, certainly, such a connected component can contain at most one element that is not in $s(\Omega)$. With that in mind, we can use more precise terminology to describe the connected components of elements of $\J(\Omega)$.

\begin{definition}
Let $\, \Omega$ be a set, $s \in \J(\Omega)$, and $\, \Sigma \subseteq \Omega$ a connected component of $s$. In this context we refer to $\, \Sigma$ as a \emph{cycle} $($or \emph{orbit}$)$ of $s$.

We say that $\, \Sigma$ is a \emph{forward cycle} $($or \emph{forward orbit} or \emph{right ray}$)$ of $s$ if $\, \Sigma$ is infinite and there is an element $\alpha \in \Sigma \setminus s(\Omega)$. In this case, we refer to $\alpha$ as the \emph{initial element} of $\, \Sigma$. If $\, \Sigma$ is infinite but not a forward cycle, then we refer to it as an \emph{open cycle} $($or \emph{open orbit} or \emph{double ray}$)$.

Given two cycles $\, \Sigma_1$ and $\, \Sigma_2$ of $s$, we say that $\, \Sigma_1$ and $\, \Sigma_2$ are \emph{of the same type} if $\, |\Sigma_1| = |\Sigma_2|$, and, in case  $\, |\Sigma_1| = |\Sigma_2| = \aleph_0$, both $\, \Sigma_1$ and $\, \Sigma_2$ are either forward or open.
\end{definition}

The next lemma will help us characterize $\sim_p$ and $\sim_p^1$ in $\J(\Omega)$ for arbitrary $\Omega$.

\begin{lemma}\label{equiv-decomp}
Let $\, \Omega$ be a set, and $s,t \in \J(\Omega)$. For each cycle $\, \Sigma$ of $ts$, let \[\Sigma^s = \left\{ \begin{array}{ll}
s(\Sigma) & \text{if $\, \Sigma$ is a finite or open cycle}\\
s(\Sigma) \cup t^{-1}(\alpha) & \text{if $\, \Sigma$ is a forward cycle with initial element } \alpha
\end{array}\right..\] 
Then sending $\, \Sigma \mapsto \Sigma^s$ defines a bijection between the set of cycles of $ts$ and the set of cycles of $st$. Moreover, in each case $\, \Sigma$ and $\, \Sigma^s$ are of the same type.
\end{lemma}

\begin{proof}
Suppose that $\Sigma$ is an open cycle of $ts$. Then, using Lemma~\ref{conn-comp-lemm}(2), we can write $\Sigma = \{\alpha_i \mid i \in \Z\}$, where $ts(\alpha_i) = \alpha_{i+1}$ for all $i \in \Z$. Hence $\Sigma^s = \{s(\alpha_i) \mid i \in \Z\}$, and $st(s(\alpha_i)) = s(\alpha_{i+1})$ for all $i \in \Z$. It follows that $\Sigma^s$ is an open cycle of $st$.

Next suppose that $\Sigma$ is a finite cycle of $ts$. Since $ts$ is injective, we can write $\Sigma = \{\alpha_0, \dots, \alpha_n\}$ for some $n \in \N$, where $ts(\alpha_i) = \alpha_{i+1\mod n}$ for all $0 \leq i < n$. The same computation as above shows that $\Sigma^s = \{s(\alpha_0), \dots, s(\alpha_n)\}$ is a finite cycle of $st$, of the same cardinality as $\Sigma$.

Finally, suppose that $\Sigma$ is a forward cycle of $ts$. Write $\Sigma = \{\alpha_0, \alpha_1, \dots\}$, where $\alpha_0$ is the initial element, and $ts(\alpha_i) = \alpha_{i+1}$ for all $i \in \N$. Then 
\[\Sigma^s = \{\beta, s(\alpha_0), s(\alpha_1), \dots\},\]
where $t^{-1}(\alpha_0) = \{\beta\}$ in case $t^{-1}(\alpha_0) \neq \emptyset$ (relying on the fact that $t$ is injective), and it is understood that $\beta$ is omitted from $\Sigma^s$ if $t^{-1}(\alpha_0) = \emptyset$. Then $st(\beta) = s(\alpha_0)$, and $st(s(\alpha_i)) = s(\alpha_{i+1})$ for all $i \in \N$. To conclude that $\Sigma^s$ is a forward cycle of $st$, it suffices to show that if $t^{-1}(\alpha_0) \neq \emptyset$, then $(st)^{-1}(\beta) = \emptyset$. Thus suppose that $t^{-1}(\alpha) \neq \emptyset$ and there exists $\gamma \in \Omega$ such that $st(\gamma) = \beta$. Then $ts(t(\gamma)) = t(\beta) = \alpha_0$, which contradicts $\alpha_0$ being the initial element in the forward cycle $\Sigma$ of $ts$. Hence if $t^{-1}(\alpha) \neq \emptyset$, then the single element of $t^{-1}(\alpha)$ is initial in $\Sigma^s$. Therefore $\Sigma^s$ is a forward cycle of $st$.

We have shown that for each cycle $\Sigma$ of $ts$,  $\Sigma^s$ is a cycle of $st$ of the same type. Now, for each cycle $\Gamma$ of $st$, define
\[\Gamma^t = \left\{ \begin{array}{ll}
t(\Gamma) & \text{if $\, \Gamma$ is a finite or open cycle}\\
t(\Gamma) \cup s^{-1}(\alpha) & \text{if $\, \Gamma$ is a forward cycle with initial element } \alpha
\end{array}\right..\] 
Then, by symmetry, each $\Gamma^t$ is a cycle of $ts$, of the same type as $\Gamma$. To conclude the proof it suffices to show that $(\Sigma^{s})^{t} = \Sigma$ for each cycle $\Sigma$ of $ts$, and $(\Gamma^{t})^{s} = \Gamma$ for each cycle $\Gamma$ of $st$. Again, given the symmetry of the situation, we shall only treat the cycles of $ts$.

Let $\Sigma$ be a cycle of $ts$. If $\Sigma$ is finite or open, then $(\Sigma^{s})^{t} = ts(\Sigma) = \Sigma$. Hence we may assume that $\Sigma$ is a forward cycle, and write $\Sigma = \{\alpha_0, \alpha_1, \dots\}$, where $\alpha_0$ is the initial element, and $ts(\alpha_i) = \alpha_{i+1}$ for all $i \in \N$. As before,
\[\Sigma^{s} = \{\beta, s(\alpha_0), s(\alpha_1), \dots\},\]
where $t^{-1}(\alpha_0) = \{\beta\}$ in case $t^{-1}(\alpha_0) \neq \emptyset$, and $\beta$ is omitted otherwise. If $t^{-1}(\alpha_0) = \emptyset$, then $s^{-1}(s(\alpha_0)) = \{\alpha_0\}$ gives
\[(\Sigma^{s})^{t} = \{\alpha_0, ts(\alpha_0), ts(\alpha_1), \dots\} = \{\alpha_0, \alpha_1, \dots\} = \Sigma.\] 
So we may assume that $t^{-1}(\alpha_0) \neq \emptyset$. As before, it is easy to see that $s^{-1}(\beta) = \emptyset$, since otherwise there would exist $\gamma \in \Omega$ such that $ts(\gamma) = \alpha_0$. Thus 
\[(\Sigma^{s})^{t} = \{t(\beta), ts(\alpha_0), ts(\alpha_1), \dots\} = \{\alpha_0, \alpha_1, \dots\} = \Sigma.\] Therefore, in all cases, $(\Sigma^{s})^{t} = \Sigma$, as desired.
\end{proof}

We are now ready to generalize the aforementioned characterization of $\sim_p^1$ in $\J(\Omega)$, with countable $\Omega$, from~\cite{AKKM1}, and extend~\cite[Theorem 5.3]{K}, which characterizes $\sim_n$, while also giving an alternative proof of that result.

\begin{theorem} \label{equiv-rel}
Let $\, \Omega$ be a set, and $s,t \in \J(\Omega)$. Then the following are equivalent.
\begin{enumerate}
\item[$(1)$] $s \sim_p t$.
\item[$(2)$] $s \sim_p^1 t$.
\item[$(3)$] $s \sim_n t$.
\item[$(4)$] $s = ptp^{-1}$ for some $p \in \B(\Omega)$.
\item[$(5)$] $E(s) \cong E(t)$.
\item[$(6)$] There is a bijection between the set of cycles of $s$ and the set of cycles of $t$, that sends each cycle to one of the same type.
\end{enumerate}
\end{theorem}

\begin{proof}
(1) $\Rightarrow$ (6) First suppose that $s \sim_p^1 t$. Then, by Lemma~\ref{equiv-decomp}, there is a bijection between the sets of cycles of $s$ and $t$, which preserves the cycle types. Since the existence of such bijections is transitive, it follows that if $s \sim_p t$, then (6) holds.

(6) $\Rightarrow$ (5) This follows from the easy observation that two cycles of the same type are isomorphic as graphs.

(5) $\Rightarrow$ (4) Suppose that $f: E(s) \to E(t)$ is a graph isomorphism. Then, in particular, $f \in \B(\Omega)$. Now let $\alpha, \beta \in \Omega$ be such that $s(\alpha) = \beta$. Then $t(f(\alpha)) = f(\beta)$, and so $f^{-1}tf(\alpha) = \beta$. Since $\alpha \in \Omega$ was arbitrary, we conclude that $s = ptp^{-1}$, where $p = f^{-1}$.

(4) $\Rightarrow$ (3) If $s = ptp^{-1}$ for some $p \in \B(\Omega)$, then it follows immediately from Definition~\ref{sim-def2} that $s \sim_n t$.

(3) $\Rightarrow$ (2) By Proposition~\ref{conj-compar} and the subsequent remark, $\sim_n \, \subseteq \, \sim_p^1$ in any semigroup.

(2) $\Rightarrow$ (1) This follows immediately from Definition~\ref{sim-def1}.
\end{proof}

Unlike $\T(\Omega)$, $\I(\Omega)$, and $\PT(\Omega)$, there does not seem to be a classification of the congruences of $\J(\Omega)$ in the literature. However, we can use other results about this semigroup to describe $\sim_s$.

\begin{theorem} \label{inj-sym-thrm}
Let $\, \Omega$ be a set, and $s,t \in \J(\Omega)$. If $\, \Omega$ is infinite, then $s \sim_s t$ if and only if $\, |\Omega \setminus s(\Omega)| = |\Omega \setminus t(\Omega)|$. If $\, \Omega$ is finite, and hence $\J(\Omega) =  \B(\Omega)$, then $s \sim_s t$ if and only if $st^{-1}$ is an even permutation. 
\end{theorem}

\begin{proof}
If $\Omega$ is finite, then $\J(\Omega) =  \B(\Omega)$, and $[\B(\Omega), \B(\Omega)]$ is the alternating subgroup of $\B(\Omega)$, by~\cite[Theorem 1]{O}. So, in this case, $s \sim_s t$ if and only if $st^{-1}$ is an even permutation, Corollary~\ref{group-sym}. We may therefore assume that $\Omega$ is infinite.

Suppose that $s \sim_s^1 t$. Then there exist $n \in \Z^+$, $p_1, \dots, p_n \in \J(\Omega)$, and $f \in \B(\{1, \dots, n\})$ such that $s=p_1\cdots p_{n}$ and $t = p_{f(1)}\cdots p_{f(n)}$. It is well-known (see, e.g., \cite[Lemma 5]{M1}) and easy to show that 
\[\sum_{i=1}^n |\Omega \setminus p_i(\Omega)| = |\Omega \setminus p_1\cdots p_{n}(\Omega)|\]
for any $p_1, \dots, p_n \in \J(\Omega)$. Hence 
\[|\Omega \setminus s(\Omega)| = \sum_{i=1}^n |\Omega \setminus p_i(\Omega)| = |\Omega \setminus t(\Omega)|.\]
It follows, by the transitivity of equality, that if $s \sim_s t$, then $|\Omega \setminus s(\Omega)| = |\Omega \setminus t(\Omega)|$.

Conversely, suppose that $|\Omega \setminus s(\Omega)| = |\Omega \setminus t(\Omega)|$. If this cardinal is $0$, then $s,t \in \B(\Omega)$. In this case $s,t \in [\B(\Omega), \B(\Omega)]$, by~\cite[Theorem 6]{O}, and hence $s \sim_s^1 t$, by Corollary~\ref{group-sym}. We may therefore assume that $s,t \in \J(\Omega)\setminus \B(\Omega)$. It is easy to see that there is a one-to-one correspondence between $\Omega \setminus s(\Omega)$ and forward cycles of $s$ (see, e.g., \cite[Lemma 4]{M1}). Therefore $s$, and likewise $t$, must have at least one forward, and hence infinite, cycle. 

Now suppose that $\Omega$ is countably infinite, and let $p \in \B(\Omega)$ be any element having at least one infinite cycle. Then, according to~\cite[Theorem 9]{M1}, there exist $q,r \in \B(\Omega)$ such that $s=qtq^{-1}rpr^{-1}$. Hence $s \sim_s^1 tp$. Likewise, $t \sim_s^1 tp$, and so $s \sim_s t$. We may therefore suppose that $\Omega$ is uncountable (and that $s,t \in \J(\Omega)\setminus \B(\Omega)$).

For each $p \in \J(\Omega)$ let $\Upsilon_p$ denote the (cardinal) number the forward cycles of $p$, let $\{\Sigma^p_{\alpha} \subseteq \Omega \mid \alpha \in \Upsilon_p\}$ be the set of the forward cycles of $p$, let $\Phi_p = \bigcup_{\alpha \in \Upsilon_p} \Sigma^p_{\alpha}$, and let $\Xi_p = \Omega \setminus \Phi_p$. Since $|\Omega \setminus s(\Omega)| = |\Omega \setminus t(\Omega)|$, as mentioned above, we must have $\Upsilon_s = \Upsilon_t$, and hence $|\Phi_s| = |\Phi_t|$ (as each forward cycle has cardinality $\aleph_0$). If $|\Phi_s| = |\Phi_t| < |\Omega|$, then $|\Xi_s| = |\Omega| = |\Xi_t|$. In this case, we can find $p \in \B(\Omega)$ that takes $\Sigma^s_{\alpha}$ bijectively to $\Sigma^t_{\alpha}$, in such a way that $p^{-1}tp$ and $s$ agree on $\Sigma^s_{\alpha}$, for each $\alpha \in \Upsilon_s = \Upsilon_t$. Then $\Phi_{p^{-1}tp} = \Phi_s$ and $\Xi_{p^{-1}tp} = \Xi_s$. Since $p^{-1}tp$ and $s$ act as permutations on $\Xi_s$, we can find a $q \in \B(\Omega)$ that acts as the identity on $\Phi_s$, such that $s = p^{-1}tpq$. As discussed before, by~\cite[Theorem 6]{O} (which says that $\B(\Omega) = [\B(\Omega), \B(\Omega)]$), $q \sim_s^1 1$. Thus $s =   p^{-1}tpq \sim_s^1 tq \sim_s^1 t$, by Lemma~\ref{Leroy-lemma2}(1). We may therefore assume that $|\Phi_s| = |\Phi_t| = |\Omega|$. 

Since $\Omega$ is uncountable, and each forward cycle is countable, we have $\Upsilon_s = |\Omega| = \Upsilon_t$. Write $\Xi_s = \bigcup_{\alpha \in \Upsilon_s} \Gamma^s_{\alpha}$ and $\Xi_t = \bigcup_{\alpha \in \Upsilon_t} \Gamma^t_{\alpha}$, where each union is disjoint, and each $\Gamma^s_{\alpha}$ and $\Gamma^t_{\alpha}$ is countable (possibly empty), and consists of finite or open cycles of $s$, respectively $t$ (with every such cycle being contained in $\Xi_s$, respectively $\Xi_t$). So 
\[\Omega = \bigcup_{\alpha \in \Upsilon_s} (\Sigma^s_{\alpha} \cup \Gamma^s_{\alpha}) = \bigcup_{\alpha \in \Upsilon_t} (\Sigma^t_{\alpha} \cup \Gamma^t_{\alpha}),\] 
with all the unions disjoint, and
\[|\Sigma^s_{\alpha} \cup \Gamma^s_{\alpha}| = \aleph_0 = |\Sigma^t_{\alpha} \cup \Gamma^t_{\alpha}|\]
for each $\alpha \in \Upsilon_s = \Upsilon_t$. Thus we can find $p \in \B(\Omega)$ such that $p(\Sigma^s_{\alpha} \cup \Gamma^s_{\alpha}) = \Sigma^t_{\alpha} \cup \Gamma^t_{\alpha}$ for each $\alpha$. Let $s_{\alpha}$, respectively $p_{\alpha}$, denote the restriction of $s$, respectively $p$, to $\Sigma^s_{\alpha} \cup \Gamma^s_{\alpha}$, and let $t_{\alpha}$ denote the restriction of $t$ to $\Sigma^t_{\alpha} \cup \Gamma^t_{\alpha}$, for each $\alpha \in \Upsilon_s = \Upsilon_t$. Then $s_{\alpha}, p_{\alpha}^{-1}t_{\alpha}p_{\alpha} \in \J(\Sigma^s_{\alpha} \cup \Gamma^s_{\alpha})\setminus \B(\Sigma^s_{\alpha} \cup \Gamma^s_{\alpha})$, and 
\[|(\Sigma^s_{\alpha} \cup \Gamma^s_{\alpha}) \setminus s_{\alpha}(\Sigma^s_{\alpha} \cup \Gamma^s_{\alpha})| = 1 = |(\Sigma^s_{\alpha} \cup \Gamma^s_{\alpha}) \setminus p_{\alpha}^{-1}t_{\alpha}p_{\alpha}(\Sigma^s_{\alpha} \cup \Gamma^s_{\alpha})|\]
for each $\alpha$. Hence, by the countable $\Omega$ case (using \cite[Theorem 9]{M1}), there exist $q_{\alpha}, r_{\alpha}, x_{\alpha} \in \B(\Sigma^s_{\alpha} \cup \Gamma^s_{\alpha})$ such that 
\[s_{\alpha}=q_{\alpha}(p_{\alpha}^{-1}t_{\alpha}p_{\alpha})q_{\alpha}^{-1}r_{\alpha}x_{\alpha}r_{\alpha}^{-1}\] 
for each $\alpha$. Letting $q,r,x \in \B(\Omega)$ be such that the restriction to each $\Sigma^s_{\alpha} \cup \Gamma^s_{\alpha}$ is $q_{\alpha}$, $r_{\alpha}$, $x_{\alpha}$, respectively, we have $s=q(p^{-1}tp)q^{-1}rxr^{-1}$. As before, $x \sim_s^1 1$, and so $s \sim_s t$.
\end{proof}

According to~\cite[Theorem 7.6]{AKM}, for any set $\Omega$ and any $s,t \in \J(\Omega)$, we have $s \sim_o t$ if and only if $s \sim_c t$ if and only if $s$ and $t$ have the same (cardinal) number of infinite cycles, open cycles, and finite cycles of each size. So for $\J(\Omega)$ each of the relations $\sim_n \, = \, \sim_p$, $\sim_o \, = \, \sim_c$, and $\sim_s$ conveys a very natural piece of information about the elements.

\section{Surjective Function Semigroups}

Given a set $\Omega$, we denote by $\SR(\Omega)$ the monoid of all surjective functions from $\Omega$ to $\Omega$. If $\Omega$ is finite, then $\SR(\Omega) = \B(\Omega) = \J(\Omega)$, and so the relations in Definitions~\ref{sim-def1}, \ref{symm-conj-def}, and~\ref{sim-def2} can be classified completely (see Theorems~\ref{equiv-rel} and~\ref{inj-sym-thrm}, and use the fact that all the relations in Definitions~\ref{sim-def1} and~\ref{sim-def2} reduce to group-conjugacy in $\B(\Omega)$). For arbitrary $\Omega$, the relation $\sim_n$ on $\SR(\Omega)$ is described in~\cite[Theorem 2.39]{ABKKMM}. It appears, however, that other sorts of conjugacy relations and congruences on this semigroup have not been studied much before in the infinite case. 

It seems that a full classification of $\sim_p $-equivalence classes or $\sim_s$-equivalence classes in infinite $\SR(\Omega)$ would take a fair amount of work to obtain, particularly since, unlike $\T(\Omega)$, $\I(\Omega)$, $\PT(\Omega)$, and $\J(\Omega)$, there is not a wealth of literature about $\SR(\Omega)$ to exploit. So we shall not attempt such classifications here. However, using Theorem~\ref{least-comm}, we can quickly obtain a rough idea about what $\sim_s$ looks like in $\SR(\Omega)$. In particular, unlike the case of $\T(\Omega)$, $\I(\Omega)$, and $\PT(\Omega)$ (Proposition~\ref{full-transf-sim}), the relation $\sim_s$ is very much nontrivial for infinite $\SR(\Omega)$.

We begin with some notation and a technical lemma.

\begin{definition}
Let $\, \Omega$ be a set, and $s \in \T(\Omega)$. Define
\[N(s) = \{\alpha \in \Omega \mid \exists \beta \in \Omega\setminus \{\alpha\} \ (s(\alpha) = s(\beta))\},\]
\[C(s) = \{\alpha \in \Omega \mid |s^{-1}(\alpha)| >1\},\]
and
\[m(s) = \mathrm{sup}\{|s^{-1}(\alpha)| \mid \alpha \in \Omega\}.\]
We say that $s$ \emph{achieves} $m(s)$ if $m(s) = |s^{-1}(\alpha)|$ for some $\alpha \in \Omega$.
\end{definition}

\begin{lemma} \label{surj-sym-tech-lem1}
Let $\, \Omega$ be a set, and $s,t \in \SR(\Omega)$. Then the following hold.
\begin{enumerate}
\item[$(1)$] $|N(st)| = |N(t)| + |N(s)\setminus C(t)|$.
\item[$(2)$] $|C(st)| = |C(s)| + |C(t) \setminus N(s)|$.
\item[$(3)$] If either $N(s)$ or $N(t)$ is infinite, then $\, |N(st)| = \max\{|N(s)|,|N(t)|\}$.
\item[$(4)$] $m(s), m(t) \leq m(st) \leq m(s)\cdot m(t)$.
\item[$(5)$] If either $m(s)$ or $m(t)$ is infinite, then $m(st) = \max\{m(s),m(t)\}$.
\item[$(6)$] If either $s$ or $t$ has a preimage of size $m(st)$, then so does $st$. If $m(st)$ is a regular cardinal, then the converse holds as well. 
\end{enumerate}
\end{lemma}

\begin{proof}
(1) Since $t$ is surjective, we have 
\[N(st) = N(t) \cup t^{-1}(N(s)) = N(t) \cup t^{-1}(N(s)\setminus C(t)),\]
where the last union is disjoint. Since $t^{-1}$ is an injective function on $\Omega \setminus C(t)$, we have 
\[|t^{-1}(N(s)\setminus C(t))| = |N(s)\setminus C(t)|,\] and so the desired formula follows.

(2) Since $t$ is surjective, we have
\[C(st) = C(s) \cup s(C(t)) = C(s) \cup s(C(t) \setminus N(s)),\]
where the last union is disjoint. Since $s$ is injective on $\Omega \setminus N(s)$, the desired formula follows.

(3) Again using $t$ being surjective, we have $|N(s)| \leq |N(st)|$. So (1) implies that 
\[|N(s)|, |N(t)| \leq |N(st)| \leq |N(s)| + |N(t)|.\]
If either $N(s)$ or $N(t)$ is infinite, then $|N(s)| + |N(t)| = \max\{|N(s)|,|N(t)|\}$. So the claim follows from the Cantor-Bernstein theorem.

(4) For any $\alpha \in \Omega$, we have 
\begin{equation}\label{eq-1} \tag{\ddag}
(st)^{-1}(\alpha) = \bigcup_{\beta \in s^{-1}(\alpha)} t^{-1}(\beta),
\end{equation} 
and so 
\[m(st) =  \mathrm{sup}\bigg\{\bigg|\bigcup_{\beta \in s^{-1}(\alpha)} t^{-1}(\beta)\bigg| \mid \alpha \in \Omega\bigg\} \leq m(s)\cdot m(t).\]
Since $s$ and $t$ are surjective, and hence $t^{-1}(\beta) \neq \emptyset \neq s^{-1}(\alpha)$ for all $\alpha, \beta \in \Omega$, we also have
\[m(s), m(t) \leq  \mathrm{sup}\bigg\{\bigg|\bigcup_{\beta \in s^{-1}(\alpha)}t^{-1}(\beta)\bigg| \mid \alpha \in \Omega\bigg\} = m(st).\]

(5) If either $m(s)$ or $m(t)$ is infinite, then $m(s) \cdot m(t) = \max\{m(s),m(t)\}$. So the desired conclusion follows from (4) and the Cantor-Bernstein theorem.

(6) For any $\alpha \in \Omega$, we have $|t^{-1}(\alpha)| \leq |(st)^{-1}(s(\alpha))|$. So if $|t^{-1}(\alpha)| = m(st)$ for some $\alpha \in \Omega$, then $|(st)^{-1}(s(\alpha))| = m(st)$. Next, for any $\alpha \in \Omega$, we have $|s^{-1}(\alpha)| \leq |(st)^{-1}(\alpha)|$, since $t$ is surjective. Hence, if $|s^{-1}(\alpha)| = m(st)$ for some $\alpha \in \Omega$, then $|(st)^{-1}(\alpha)| = m(st)$. 

Now suppose that $|(st)^{-1}(\alpha)| = m(st)$ for some $\alpha \in \Omega$. If $m(st)$ is regular, then either $|s^{-1}(\alpha)| = m(st)$, or $|t^{-1}(\beta)| = m(st)$ for some $\beta \in s^{-1}(\alpha)$, by~(\ref{eq-1}) and (4).
\end{proof}

The next result gives a partial description of $\sim_s$ in $\SR(\Omega)$.

\begin{theorem} \label{surj-sym}
Let $\, \Omega$ be a countably infinite set, and $s,t \in \SR(\Omega)$. Write $s \approx t$ if any of the following conditions holds.
\begin{enumerate}
\item[$(1)$] $|N(s)|,|N(t)|<\aleph_0$, and $\, |N(s)|-|C(s)| = |N(t)|-|C(t)|$.
\item[$(2)$] $|N(s)| = |N(t)| = \aleph_0$, and $m(s),m(t) < \aleph_0$.
\item[$(3)$] $m(s) = m(t) = \aleph_0$, but $s$ and $t$ do not achieve $m(s) = m(t)$.
\item[$(4)$] $m(s) = m(t) = \aleph_0$, and $s$ and $t$ achieve $m(s) = m(t)$.
\end{enumerate}
Then $\, \approx$ is a congruence, and $\sim_s \, \subseteq \, \approx$.
\end{theorem}

\begin{proof}
Let $S = \N \cup \{\infty_1, \infty_2, \infty_3\}$, and extend in a commutative fashion the addition from the semigroup $(\N,+)$ to $S$, by letting $s + \infty_i = \infty_i$ for all $s \in \N$ and $i \in \{1,2,3\}$, and letting $\infty_i + \infty_j = \infty_{\max\{i,j\}}$ for all $i \in \{1,2,3\}$. With this operation, $S$ is clearly a commutative semigroup. Define $f: \SR(\Omega) \to S$ by
\[f(s) = \left\{ \begin{array}{cl}
|N(s)| - |C(s)| & \text{if } |N(s)|<\aleph_0\\
\infty_1 & \text{if } |N(s)| = \aleph_0 \text{ and } m(s) < \aleph_0\\
\infty_2 & \text{if } m(s) = \aleph_0 \text{ and } s \text{ does not achieve } m(s)\\
\infty_3 & \text{if } m(s) = \aleph_0 \text{ and } s \text{ achieves } m(s)\\
\end{array}\right..\] 
We shall show that $f$ is a semigroup homomorphism. Since $\approx$ is clearly the kernel of $f$, it follows from this that $\approx$ is a congruence (see, e.g., \cite[Theorem 1.5.2]{H}). Since $S$ commutative, Theorem~\ref{least-comm} then implies that $\sim_s \, \subseteq \, \approx$.

Let $s,t \in S$. If $|N(s)|,|N(t)|<\aleph_0$, then, by Lemma~\ref{surj-sym-tech-lem1}(1,2),
\[|N(st)| - |C(st)| = |N(t)| + |N(s)\setminus C(t)| - |C(s)| - |C(t) \setminus N(s)|\]
\[= |N(t)| + |N(s)| - |N(s) \cap C(t)| - |C(s)| - |C(t)| + |C(t) \cap N(s)|\] 
\[= |N(s)| - |C(s)| + |N(t)| - |C(t)|,\] 
and so 
\[f(st) = |N(st)| - |C(st)| = |N(s)| - |C(s)| + |N(t)| - |C(t)| = f(s)+f(t).\]

Next suppose that $|N(t)| <\aleph_0$ but $|N(s)| = \aleph_0$. Then, by Lemma~\ref{surj-sym-tech-lem1}(3), $|N(st)| = |N(s)|$, by Lemma~\ref{surj-sym-tech-lem1}(4), $m(st) = \aleph_0$ if and only if $m(s) = \aleph_0$, and, by Lemma~\ref{surj-sym-tech-lem1}(6), in case $m(st) = \aleph_0$, $st$ has an infinite preimage if and only if $s$ does (since $\aleph_0$ is regular). Writing $f(s) = \infty_i$ for some $i \in \{1,2,3\}$, it follows that 
\[f(st) = \infty_i = \infty_i + f(t) = f(s)+f(t).\]
Very similar considerations show that if $|N(t)| = \aleph_0$ but $|N(s)| < \aleph_0$, then $f(st) = f(s)+f(t)$.

We may therefore assume that $|N(s)| = \aleph_0 = |N(t)|$. Then, by Lemma~\ref{surj-sym-tech-lem1}(3), $|N(st)| = \aleph_0$, by Lemma~\ref{surj-sym-tech-lem1}(4), $m(st) = \aleph_0$ if and only if either $m(s) = \aleph_0$ or $m(t) = \aleph_0$, and, by Lemma~\ref{surj-sym-tech-lem1}(6), in case $m(st) = \aleph_0$, $st$ has an infinite preimage if and only if either $s$ or $t$ does (since $\aleph_0$ is regular). Writing $f(s) = \infty_i$ and $f(t) = \infty_j$ for some $i,j \in \{1,2,3\}$, it follows that 
\[f(st) = \infty_{\max\{i,j\}} = f(s) + f(t).\]
Hence $f$ is a semigroup homomorphism, as claimed.
\end{proof}

We note that the proof above does not really rely on $\Omega$ being countable--just on $|\Omega|$ being regular. Rather this assumption was imposed, since otherwise there would clearly be more possibilities for the values of $|N(s)|$, $|N(t)|$, $m(s)$, and $m(t)$. To extend the theorem to $\Omega$ of arbitrary cardinality in a nontrivial way, one would need to consider not only all possible infinite values of $m(s) \leq |N(s)|$ and $m(t) \leq |N(t)|$ that are $\leq |\Omega|$, but also quantify the prevalence of preimages of $s$ and $t$ of various infinite cardinalities.

\section*{Appendix: Traces on Semigroup Rings}

As we have discussed, $\sim_s$ has a special relationship with semigroup homomorphisms--it describes precisely what must be related by a homomorphism, for the image to be commutative and as large as possible (Theorem~\ref{least-comm}). It turns out that $\sim_p$ has a similar relationship with certain trace functions on semigroup rings, which we briefly discuss next. 

From now on we assume all rings to be unital. The following definition is taken from~\cite{MV}.

\begin{definition} \label{trace-def}
Let $R$ and $T$ be rings, and let $f: R \to T$ be an additive function $\, ($i.e, $f(s+t)=f(s)+f(t)$ for all $s,t \in R$$)$. 

If $R$ and $T$ are $C$-algebras, for some commutative ring $C$, then we say that $f$ is $C$-\emph{linear} in case $f(rs)=rf(s)$ for all $s \in R$ and $r \in C$.

We say that $f$ is a $T$-\emph{valued trace on} $R$ if $f(st)=f(ts)$ for all $s,t \in R$. If $f$ is a trace on $R$, then we say that $f$ is \emph{minimal} if $f(s)=0$ implies that $s$ is a sum of additive commutators, for all $s \in R$.
\end{definition}

\begin{lemma} \label{tr-lem}
Let $R$ and $T$ be rings, and $f : R \to T$ a trace. For all $s,t \in R$, if $s \sim_p t$, then $f(s) = f(t)$.
\end{lemma}

\begin{proof}
Let $s,t \in R$, and suppose that $s \sim_p t$. Then, according to~\cite[Theorem 3.15(2)]{AlL}, $s-t$ is a sum of additive commutators in $R$. Since $f$ is an (additive) trace, it follows that $f(s-t) = 0$, and hence $f(s) = f(t)$.
\end{proof}

Let $R$ be a ring, and $S$ a semigroup with zero. We denote by $RS$ the corresponding semigroup ring, and by $\overline{RS}$ the resulting \emph{contracted semigroup ring}, where the zero of $S$ is identified with the zero of $RS$. That is, $\overline{RS} = RS/I$, where $I = \{x\cdot 0_S \in RS \mid x \in R\}$ is the ideal of $RS$ generated by the zero $0_S$ of $S$. An arbitrary element of $\overline{RS}$ can be represented as $\sum_{s \in S\setminus \{0\}} a_ss$, where $a_s \in R$, and all but finitely many of the $a_s$ are zero. 

The second statement in the next proposition effectively says that $\sim_p$ relates exactly the elements of a semigroup $S$ that must be identified by every linear trace on $\overline{RS}$. 

\begin{proposition} \label{tr-prim}
Let $R$ be a commutative ring, $T$ an $R$-algebra, $S$ a semigroup with zero, and $f : \overline{RS} \to T$ an $R$-linear function.
\begin{enumerate}
\item[$(1)$] The map $f$ is a trace if and only if $s \sim_p t$ implies that $f(s) = f(t)$ for all $s,t \in S$.
\item[$(2)$] Suppose that $f$ is a minimal trace. Then $s \sim_p t$ if and only if $f(s) = f(t)$, for all $s,t \in S$.
\end{enumerate}
\end{proposition}

\begin{proof}
(1) Suppose that $f$ is a trace. Then $s \sim_p t$ implies that $f(s) = f(t)$ for all $s,t \in S$, by Lemma~\ref{tr-lem}. For the converse, suppose that $s \sim_p t$ implies that $f(s) = f(t)$ for all $s,t \in S$. Then, in particular, $f(st) = f(ts)$ for all $s,t \in S$. It follows that $f(pr) = f(rp)$ for all $p,r \in \overline{RS}$, since $f$ is $R$-linear. Therefore $f$ is a trace.

(2) It is shown in~\cite[Theorem 11(2)]{MV} that if $f : \overline{RS} \to T$ is a minimal trace, then $f$ takes elements of $S$ from different $\sim_p$-equivalence classes to $R$-linearly independent elements of $T$. Thus if $f(s) = f(t)$ for some $s,t \in S$, then $s \sim_p t$. The converse follows from (1).
\end{proof}

Recall that if $R$ is a commutative ring and $n \in \Z^+$, then the ring $\M_n(R)$ of $n\times n$ matrices over $R$ is isomorphic to the contracted semigroup ring $\overline{RS}$, where 
\[S=\{e_{ij} \mid 1\leq i,j \leq n\} \cup \{0\},\] 
$e_{ij}$ are the matrix units, and multiplication is given by
\[e_{ij}\cdot e_{kl} = 
\left\{ \begin{array}{ll}
e_{il} & \text{if } j = k\\
0 & \text{if } j \neq k
\end{array}\right..\] 
Since the usual trace on $\M_n(R)$ is minimal (see, e.g., \cite[Corollary 14]{MV}), $\sim_p$ agrees with it on matrix units, by the previous proposition. It is not hard to see that $\sim_w$ does as well, but that the other relations in Definitions~\ref{sim-def1}, \ref{symm-conj-def}, and~\ref{sim-def2} do not, provided that $n \geq 2$. 

The relations $\sim_p$ and $\sim_p^1$ on matrix rings are explored in greater detail in~\cite{AbL,AlL}.

\section*{Acknowledgements}

I am grateful to the referee for a very thoughtful review and valuable suggestions, and to Gene Abrams for a helpful discussion of this material.

\medskip

\noindent Department of Mathematics, University of Colorado, Colorado Springs, CO, 80918, USA 

\noindent \emph{Email:} \href{mailto:zmesyan@uccs.edu}{zmesyan@uccs.edu}

\end{document}